\documentclass[10pt,a4paper]{amsart}

\usepackage{amsmath}
\usepackage{amsfonts}
\usepackage{amssymb}
\usepackage{amsthm}

\usepackage{comment}
\usepackage{tikz}
\usepackage{tikz-cd}
\usepackage{stmaryrd}
\usepackage{hyperref}

\usepackage[lite,alphabetic]{amsrefs}

\usepackage{bm}

\usepackage{mathtools}

\usepackage{enumitem}
\usepackage[utf8]{inputenc}
\usepackage[british]{babel}

\usepackage{csquotes}

\usepackage{embedfile}
\embedfile{main.tex}

\newcommand{\pp}{\mathfrak{p}}

\newcommand{\QQ}{\mathbb{Q}}
\newcommand{\ZZ}{\mathbb{Z}}
\newcommand{\NN}{\mathbb{N}}
\newcommand{\Qp}{\mathbb{Q}_p}
\newcommand{\Zp}{{\mathbb{Z}_p}}

\newcommand{\GL}{\mathrm{GL}}

\newcommand{\Hom}{\mathrm{Hom}}
\newcommand{\End}{\operatorname{End}}

\newcommand{\Spec}{\operatorname{Spec}}
\newcommand{\Ind}{\operatorname{Ind}}

\newcommand{\im}{\operatorname{im}}
\newcommand{\rk}{\operatorname{rk}}

\newcommand{\mcK}{\mathcal{K}}

\newcommand{\mcS}{\mathcal{S}}

\newcommand{\eG}{e\mathcal{G}}

\newcommand{\Oo}{\mathcal{O}}

\newcommand{\soc}{\mathrm{soc}}
\newcommand{\vi}{\vee}

\newcommand{\mm}{\mathfrak{m}}

\newcommand{\Frac}{\mathrm{Frac}}

\newcommand{\ann}{\mathrm{ann}}

\newcommand{\Rep}{\mathrm{Rep}}

\newcommand{\tf}{\mathrm{tf}}

\newcommand{\Fp}{{\mathbb{F}_p}}

\newcommand{\qq}{\mathfrak{q}}
\newcommand{\Aa}{\mathbb{A}}

\newcommand{\sm}{\mathrm{sm}}
\newcommand{\mcR}{\mathcal{R}}

\newcommand{\mfa}{\mathfrak{a}}

\newcommand{\Res}{\mathrm{Res}}
\newcommand{\Ifl}{I_{F_\ell}}
\newcommand{\Pfl}{I_{F_\ell}^{(p)}}
\newcommand{\Got}{ {G_{\overline{\tau}}} }

\newcommand{\Rd}{ {R^\diamond_{\overline{\rho}}  } }

\newcommand{\pr}{\mathrm{pr}}

\newcommand{\Part}{\mathrm{Part}}

\newcommand{\Rsq}{R^\square_{\overline{\rho}}}

\mathchardef\mhyphen="2D 

\newtheorem{thm}{Theorem}[subsection]
\newtheorem{lemma}[thm]{Lemma}
\newtheorem*{lemma*}{Lemma}
\newtheorem{conj}[thm]{Conjecture}
\newtheorem{prop}[thm]{Proposition}
\newtheorem{cor}[thm]{Corollary}
\newtheorem{defn}[thm]{Definition}

\newcommand{\ddelta}[1][-]{{#1}^{(n)}}
\newcommand{\ddeltap}[1][-]{ \left( #1 \right) ^{(n)}}

\author{Tibor Backhausz}
\address{Tibor Backhausz\\University College London}
\email{tib{or.}back{hau}sz.{1}5@ucl.ac.uk}

\title[Change of coefficients in the local Langlands correspondence]{Change of coefficients in the $p \neq \ell$ local Langlands correspondence for $\mathrm{GL}_n$}

\begin{document}
\maketitle
\begin{abstract}
    Let $\ell$ and $p$ be distinct primes, $n$ a positive integer, $F_\ell$ an $\ell$-adic local field of characteristic $0,$ and let $W(k)$ denote the ring of Witt vectors over an algebraically closed field of characteristic $p$. Work of Emerton-Helm, Helm and Helm-Moss defines and constructs a smooth $A[\GL_n(F_\ell)]$-module $\tilde{\pi}(\rho_A)$ for a continuous Galois representation $\rho_A : G_{F_\ell} \to \GL_n(A)$ over a $p$-torsionfree reduced complete local $W(k)$-algebra $A$ interpolating the local Langlands correspondence. However, since $\tilde{\pi}$ is not a functor, there is no clear way to speak about the local Langlands correspondence over non-reduced or finite characteristic $W(k)$-algebras. We describe two natural and reasonable variants of the local Langlands correspondence with arbitrary complete local $W(k)$-algebras as coefficients. They are isomorphic when evaluated on the universal framed deformation of a Galois representation $\overline{\rho}$ over $k$, and more generally we find a surjection in one direction. In many cases, including $n=2$ or $3,$ they both recover $\tilde{\pi}(\rho)$ when $\rho$ has coefficients in a finite extension of $W(k)[p^{-1}].$
    On the Galois side, this requires finding minimal lifts between Galois deformations.
\end{abstract}

\section{Introduction}
We first fix some notation and conventions.
Let $\ell,p$ be distinct prime numbers. 
Let $k$ be an algebraically closed field of characteristic $p.$
We write $W(k)$ for the ring of Witt vectors of $k$, and let $\mcK$ be the category of fields that are algebras over $W(k).$

Let $F_\ell$ be a mixed characteristic local field of residue characteristic $\ell$, and $n \in \NN.$
We work with the group $G_\ell:=\GL_{n}(F_\ell).$

In this paper, we consider smooth representations of an $\ell$-adic group with coefficients over $p$-adic rings -- this means that the role of the primes $\ell$ and $p$ is opposite to most of our references, but matches \cite{eh}, which is our most frequent reference.

We fix throughout a continuous Galois representation $\overline{\rho} : G_{F_\ell} \to \GL_n(k)$.
This $\overline{\rho}$ corresponds to some factor $A_{[L,\pi]}$ of the integral Bernstein centre by \cite{helm2016whittaker} (we recall the precise relationship in section 2).

Let $(R_{\overline{\rho}}^\square,\rho^\square)$ respectively be the universal framed deformation ring attached to $\overline{\rho}$ and the universal framed representation. Note that under our hypotheses, $R_{\overline{\rho}}^\square$ is reduced and flat over $W(k)$ \cite[Corollary 8.3]{helm_curtishom}.

\textbf{Conventions.} We will usually omit saying '$G_\ell$-equivariant' when discussing morphisms between smooth $G_\ell$-representations, including $R[G_\ell]$-modules for any ring $R$. For a ring $R$ and a prime
 ideal $\pp \le R$
corresponding to a point $x \in \Spec R,$ we use the notation
$$\kappa(x)=\kappa(\pp)=\Frac(R/\pp).$$
Furthermore, $\kappa^+(x)=\kappa^+(\pp)$ denotes the integral closure of $R/\pp$ inside $\kappa(x).$ However, in this paper, we only choose to use the $\kappa^+$ notation when $R/\pp$ is a discrete valuation ring, and so $\kappa^+(\pp)=R/\pp$ which we often use implicitly.\\
To improve the appearance and readability of the paper, we will use $\otimes$ with no indication of the ring to mean $\otimes_{R_{\overline{\rho}}^\square}$.

\subsection{Motivation}
Let $R$ be a Noetherian and flat complete reduced local $W(k)$-algebra $R,$ and consider a continuous representation $\rho : G_{F_\ell} \to \GL_n(R),$ or equivalently, a morphism of $W(k)$-algebras $\Rsq \to R$ satisfying $\rho=\rho^\square \otimes R.$ 

For all such rings $R$, the paper \cite{eh} due to Emerton and Helm defined (but did yet not show the existence of) a unique smooth $R[G_\ell]$-module $\tilde{\pi}(\rho),$ which is $R$-torsionfree and matches a form of the local Langlands correspondence after localisation at each minimal prime of $R.$ Since a continuous representation $\rho : G_{F_\ell} \to \GL_n(R)$ can be viewed as the family of representations of the form $\rho \otimes_R \kappa(x)$ for $x \in \Spec R,$ this is called the local Langlands correspondence in families for $\GL_n.$
One motivation for the study of this correspondence is global, in particular it is applicable to completed cohomology, as in \cite{lg} or \cite{sorensen2016local}.

$\tilde{\pi}(\rho)$ satisfies the important technical property of being \emph{co-Whittaker} as defined in \cite{helm2016whittaker}, and recalled in Definition \ref{coWh_defn} below. The structure theory of co-Whittaker modules is relatively simple. In particular, for a complete local $W(k)$-algebra $R,$ the full subcategory of smooth $R[G_\ell]$-modules that are co-Whittaker over some quotient of $R$ can be thought of as a partially ordered set with the relation \emph{admits a surjection to} as the ordering, without forgetting too much information.

The connection of $\tilde{\pi}$ to the more usual variants of the local Langlands correspondence is that $\tilde{\pi}(\rho)$ is the smooth dual of the Breuil-Schneider modification of the local Langlands correspondence if $\rho$ is defined over an algebraically closed field of characteristic $0.$

The module $\tilde{\pi}(\rho^\square \otimes R)$ can behave in subtle ways when changing $R$, i.e. its dependence on the homomorphism $\Rsq \to R$ is not functorial. A recurring pattern when dealing with $\tilde{\pi}$ is that torsion-freeness is not a very stable property, especially over rings containing zero divisors. Moreover, by definition $\tilde{\pi}$ does not distinguish between Galois representations that are isomorphic at all minimal primes of $R.$ This makes it unclear how to extend the construction to non-reduced or non-flat $R.$

\subsection{Two variants}
We describe two different constructions that have some merits to be called a version of the local Langlands correspondence over a much wider class of $W(k)$-algebras. They both assign co-Whittaker representations to a
continuous Galois representations $\rho : G_{F_\ell} \to \GL_n(R)$ where $R$ is a complete local $W(k)$-algebra.

\subsubsection*{The top-down approach.}
An approach originally considered by the authors of \cite{eh} was\footnote{according to personal communication} to designate
$$\tilde{\pi}(\rho^\square) \otimes R$$
as the representation attached to $\rho=\rho^\square \otimes R$ for any $W(k)$-algebra. This relies on the existence of $\tilde{\pi}(\rho^\square)$ which was not proven at the time of \cite{eh}. By now, $\tilde{\pi}(\rho^\square)$ is known to exist due to recent work of Helm and Moss \cite{helmmoss}. It is worthwhile to note that the actual construction of $\tilde{\pi}$ (given the main result of \cite{helmmoss}) is found in \cite{helm2016whittaker}.

An advantage of this approach is that, considering the category of pairs $(R,\alpha : \Rsq \to R)$ where $R$ is a complete $W(k)$-algebra and $\alpha$ is an algebra homomorphism, 
$$(R, \alpha) \mapsto \tilde{\pi}(\rho^\square) \otimes_\alpha R$$ defines a functor to the category $W(k)[G_\ell]$-modules, and the resulting representation is co-Whittaker over $R$ (Lemma \ref{quotientofW_lemma}).

The problem with the approach is that, at present, we cannot show that it recovers the direct definition of the local Langlands correspondence when $A$ is a field. In other words, the following is conjectural in general.
\begin{conj}
\label{the_conjecture}
Let $\pp \in \Spec \Rsq$ be a prime ideal not containing $p$. Then
$$\tilde{\pi}(\rho^\square) \otimes_{\Rsq} \kappa(\pp) \cong \tilde{\pi}(\rho^\square \otimes_{\Rsq} \kappa(\pp))$$
where $\kappa(\pp)$ is the fraction field of $\Rsq/\pp.$
\end{conj}

One of the main results of this paper is one direction of this conjecture, relying on \cite[Proposition 6.2.10]{eh} and the existence of \emph{minimal lifts} on $\Spec \Rsq[p^{-1}]$ shown in Theorem \ref{comp_finding_a_min_prime} which we prove in section 3.
\begin{thm}
\label{coeff_change_thm}
Let $B$ be a flat reduced Noetherian complete local $W(k)$-algebra equipped with a local homomorphism $\Rsq \to B.$
Then there exists a natural $G_\ell$-equivariant surjection $$\tilde{\pi}(\rho^\square) \otimes_{\Rsq} B \to \tilde{\pi}(\rho^\square  \otimes_{\Rsq} B).$$
\end{thm}
\begin{proof}
Theorem \ref{comp_finding_a_min_prime} shows that the minimal lift assumptions of \cite[Proposition 6.2.10]{eh} hold for the map $\Rsq \to B.$
\end{proof}
One surjection in Conjecture \ref{the_conjecture} then follows by applying the Theorem to $\Rsq/\pp$ and then passing to the fraction field. Note that the here, the contribution of the present paper is purely on the Galois side, through Theorem \ref{comp_finding_a_min_prime}, which might be of independent interest.

\subsubsection*{The bottom-up approach.}
Based on the semi-simple local Langlands correspondence due to Vignéras \cite{Vigneras2001} for an algebraically closed coefficient field $k$ of characteristic $p$, Emerton and Helm in \cite{eh} defined a slightly modified correspondence $\rho_k \mapsto {\overline{\pi}}(\rho_k)$ for a continuous representation $\rho_k : G_{F_\ell} \to \GL_n(k),$ where $k$ is a finite field of characteristic $p$. Its smooth $k$-dual $\tilde{\overline{\pi}}(\rho_k)$ is the universal co-Whittaker representation over $k$ admitting a surjection to all representations of the form $\rho_\Oo \otimes_\Oo k$ where $\Oo$ is a characteristic $0$ discrete valuation ring with residue field $k$, and $\rho_\Oo : G_{F_\ell} \to \GL_n(k)$ is a continuous representation satisfying $\rho_\Oo \otimes_\Oo k \cong \rho_k.$

We clarify what we mean by the word \emph{universal}.
\begin{defn}
Let $R$ be a $W(k)$-algebra, and let $(M_i)_{i \in I}$ be a collection of $R[G_\ell]$-modules. We say that $M$ is a co-Whittaker cover of $(M_i)_{i \in I},$ if it is a co-Whittaker representation of $G_\ell$ over a quotient of $R,$ and admits a surjection $f_i$ to all representations $M_i.$ It is the universal co-Whittaker cover if any co-Whittaker cover $(M',(f_i')_{i \in I})$ of $(M_i)_{i \in I}$ admits a surjection to $M$ through which the $f_i'$ all factor.
\end{defn}

Let $\mathrm{QDVR}_{k}$ denote the full subcategory of all $W(k)$-algebras which admit a (non-identity) quotient map from a discrete valuation ring with residue field $k$. Any $S \in \mathrm{QDVR}_{k}$ is Artinian, local, and each of its ideals is of the form $(\varpi)^j$ where $j \in \NN$ and $\varpi$ is any element of the maximal ideal $\mm_S$ not contained in $\mm_S^2$ if $S \neq k$ is not a field, or $\varpi=0$ if $S=k.$ We call any such element $\varpi$ a uniformiser of $S,$ as it can be lifted to a uniformiser of any discrete valuation ring admitting a surjection to $S.$

For any $S \in \mathrm{QDVR}_{k},$ let $\mathrm{DVR}_S$ denote the set of isomorphism classes of complete discrete valuation rings $\tilde{S}$ of characteristic $0$ admitting a surjection to $S.$
Analogously to the definition of $\tilde{\overline{\pi}}$ above, we set
\begin{defn}
Let $S \in \mathrm{QDVR}_{k}$ and let $\rho : G_{F_\ell} \to \GL_n(S)$ be a continuous representation.
We define $\tilde{\pi}^{\mathrm{DVR}}(\rho)$
to be the universal co-Whittaker cover of 
$$\left( \tilde{\pi}(\rho_{\tilde{S}}) \otimes_{\tilde{S}} S \right)_{\tilde{S} \in \mathrm{DVR}_S}.$$
\end{defn}
Note that $\tilde{\pi}^{\mathrm{DVR}}(\rho \otimes_R k)=\tilde{\pi}(\overline{\rho}).$

We wish to extend this definition to the class of all Noetherian complete local $W(k)$-algebras $R,$
relying on the following lemma, proved later, which uses \cite[Theorem 6.3]{helm2016whittaker}.
\begin{lemma}
\label{cWh_cover_lemma}
Let $R$ be a Noetherian $W(k)$-algebra, and let $(M_i)_{i \in I}$ be a collection of $R[G_\ell]$-modules, such that each is co-Whittaker over some quotient of $R.$ Assume moreover that the action of the integral Bernstein center on each $M_i$ factor through the same direct summand.
Then there exists a universal co-Whittaker cover of $(M_i)_{i \in I}.$
\end{lemma}

This Lemma involves quotient maps between rings, so we need to adjoin a power series variable to obtain quotient maps from non-surjective maps by the following observation for a $W(k)$-algebra $R.$
For $S \in \mathrm{QDVR}_{k},$ any local $W(k)$-algebra homomorphism $\alpha : R \to S$ extends to a map $\beta: R[[u]] \to S$ sending $u$ to a uniformiser. Any $\beta$ of this form is surjective.
\begin{defn}
For a Noetherian complete local $W(k)$-algebra $R$ with residue field $k$ and $\rho : G_{F_\ell} \to \GL_n(R),$ we define
$${\tilde{\Pi}^{\mathrm{DVR},u}}(\rho)$$
to be the universal co-Whittaker cover (over $R[[u]]$) of the collection
$\left( \tilde{\pi}^{\mathrm{DVR}}( {\rho \otimes_{\beta} S} ) \right)_{\beta}$
where $\beta$ ranges over quotient maps $\beta : R[[u]] \to S$ sending $u$ to a uniformiser of some $S \in \mathrm{QDVR}_{k}.$
\end{defn}

To get a co-Whittaker representation over a quotient of $R$ (as opposed to over a quotient of $R[[u]]$), we designate
$$\tilde{\Pi}^{\mathrm{DVR}}(\rho) := \tilde{\Pi}^{\mathrm{DVR},u}(\rho)/ u \tilde{\Pi}^{\mathrm{DVR},u}(\rho)$$
as the bottom-up version of the local Langlands correspondence for $\GL_n.$

The definition is well-behaved when passing to a quotient ring. If $R \to R'$ is a surjective homomorphism of Noetherian complete local $W(k)$-algebras, then for any $S \in \mathrm{QDVR}_{k},$ the set of quotient maps from $R[[u]]$ to $S$ taking $u$ to a uniformiser is (by restriction) naturally a subset of such maps from $R'[[u]] \to S.$ Therefore by the universal co-Whittaker cover property, there is a surjection
$$\tilde{\Pi}^{\mathrm{DVR},u}(\rho) \to \tilde{\Pi}^{\mathrm{DVR},u}(\rho \otimes_R R')$$
in this situation.

In particular, if $\rho$ is defined over some $R \in \mathrm{QDVR}_k$ with a uniformiser $\varpi$ then $\tilde{\Pi}^{\mathrm{DVR}}(\rho)$ is the universal co-Whittaker representation for the collection $$\tilde{\pi}^{\mathrm{DVR}}(\rho \otimes_R k), \tilde{\pi}^{\mathrm{DVR}}(\rho \otimes_R R/(\varpi^2) ), \dots ,\tilde{\pi}^{\mathrm{DVR}}(\rho).$$

\subsection{Comparison of the two approaches.}
Both definitions seem natural, with $\tilde{\pi}(\rho^\square) \otimes -$ being \emph{a priori} better behaved with change of rings, and $\tilde{\Pi}^{\mathrm{DVR}}$ being easier to compute for explicitly given $\rho.$ Therefore understanding their relationship precisely would be a substantial step towards defining a well-behaved local Langlands correspondence with arbitrary complete local $W(k)$-algebras as coefficient rings. Our other main result is a step in this direction.
\begin{thm}
\label{other_main_result_thm}
There exists an isomorphism
$$\tilde{\pi}(\rho^\square) \xrightarrow{\sim} {\tilde{\Pi}^{\mathrm{DVR}}}(\rho^\square),$$
and therefore a surjection
$\tilde{\pi}(\rho^\square) \otimes_{\Rsq} R \to {\tilde{\Pi}^{\mathrm{DVR}}}(\rho^\square \otimes_{\Rsq} R)$
for any quotient ring $R$ of $\Rsq.$
\end{thm}
\begin{proof}
This is the conjunction of Theorem \ref{equality_thm} and Proposition \ref{coWh_identification_prop}
\end{proof}

Given stronger assumptions, we can even show that $\tilde{\Pi}^{\mathrm{DVR}}$ recovers the original definition for the field of fractions of a Krull dimension $1$, characteristic $0$ quotient of $\Rsq.$
\begin{thm}
\label{evenstronger_thm}
Let $\pp \in \Spec \Rsq[p^{-1}]$ be such that $\Rsq/\pp$ is of Krull dimension $1.$
If the surjection $\tilde{\pi}(\rho^\square) \otimes \kappa(\pp) \to \tilde{\pi}(\rho^\square \otimes \kappa(\pp))$ is an isomorphism, then
$$\tilde{\Pi}^{\mathrm{DVR}}(\rho^\square \otimes \Rsq/\pp) \otimes \kappa(\pp) \cong \tilde{\pi}(\rho^\square \otimes \kappa(\pp)).$$
\end{thm}
\begin{proof}[See proof at end of section 4.]
\end{proof}
Note that Theorem \ref{comp_surj_iso_existence_thm} shows that in many cases, the assumption holds, in particular if $n=2$ or $n=3.$

In particular, if $\rho : G_{F_\ell} \to \GL_n(\Oo)$ where $\Oo$ is the ring of integers of some finite extension of $W(k)[p^{-1}],$ then the associated homomorphism of $W(k)$-algebras $\Rsq \to \Oo$ has kernel $\pp \in \Spec \Rsq[p^{-1}]$ with $\Rsq/\pp$ of Krull dimension $1$. If the isomorphism condition holds as well, then we deduce that 
$$\tilde{\Pi}^{\mathrm{DVR}}(\rho)[p^{-1}] \cong \tilde{\pi}(\rho)[p^{-1}],$$
hence in this case both the top-down and bottom-up approaches recover the Emerton-Helm variant of the local Langlands correspondence for $\rho$ of this form.

\subsection{Outline of the paper.} In section 2, we discuss the technical tools needed in the rest of the paper, starting with the necessary commutative algebra to obtain a sufficiently large set of surjections from $\Rsq[[u]]$ to discrete valuation rings in the sense that the set of points of $\Spec \Rsq$ corresponding to their kernels is Zariski dense. We continue by recalling the theory of co-Whittaker modules from \cite{helm2016whittaker}, building up to prove Lemma \ref{cWh_cover_lemma}. We then proceed to relate $\tilde{\Pi}^{\mathrm{DVR}}(\rho^\square)$ to the ring $\Rsq[[u]].$

Section 3 contains all the Galois theory in the present work, and is dedicated to proving Theorem \ref{comp_finding_a_min_prime}. We review the necessary Galois deformation theory, including a decomposition based on $p$-wild inertia due to \cite{CHT}, as well as the language of \emph{pseudo-framed} representations from \cite{helm_curtishom}. The claim we need is similar in spirit to the result of section 2.4.4 of \cite{CHT} saying that the 'minimally ramified' deformation condition is liftable. However, we wish to find minimal lifts to an irreducible component (in the sense of \cite{eh}) of Galois representations that might not be themselves 'minimally ramified' (compared with the fixed residue representation) in the language of \cite{CHT}. We combine these ideas with the well-known correspondence between conjugacy classes of nilpotent matrices and partitions of integers, which induces a relationship between Zariski closure and the dominance order on partitions.

The final section is dedicated to proving the isomorphism in Theorem \ref{other_main_result_thm} and Theorem \ref{evenstronger_thm}. In both cases, we establish surjections in both directions.

\subsection{Related work and further directions}
The (preprint) \cite{disegni2018local} by Disegni gives a geometric treatment of $\tilde{\pi}$ (there denoted simply by $\pi$) with a very slightly different formalism, over $K$-schemes for any characteristic $0$ field $K,$ starting from Weil-Deligne representations instead of Galois representations, and goes on to apply the construction to the theory of $L$-functions.

One reason for distinguishing discrete valuation rings as coefficient rings is because, for a smooth representation of a pro-finite group over a discretely valued field, much is already known about the reduction to the residue field by finding an invariant lattice and reducing it, and taking the semi-simplification (which is independent of the choice of lattice) of the resulting representation. In particular, in the context of inertial types for $\GL_n$, Shotton \cite{shotton} proves the $\ell \neq p$ Breuil--Mézard conjecture for (in our notation) $p>2$, meaning that the reduction in this sense commutes with the reduction of irreducible components of $\Rsq.$ This may be used to determine the modulo $p$ inertial types appearing in $\tilde{\Pi}^{\mathrm{DVR}}$.

\subsubsection*{Acknowledgements.}
The author would like to thank David Helm for many helpful discussions that influenced the development of this paper, most importantly for suggesting that $\tilde{\Pi}_{\overline{\rho}}/u \tilde{\Pi}_{\overline{\rho}}$ might in fact be equal to $\tilde{\pi}(\rho^\square),$ which became a crucial result of this work. I would also like to thank Kevin Buzzard for asking the question which resulted in this paper as well as helpful discussions throughout its development. This work was supported by the Engineering and Physical Sciences Research Council [EP/L015234/1], through the EPSRC Centre for Doctoral Training in Geometry and Number Theory (The London School of Geometry and Number Theory), University College London.

\section{Definitions and first properties}
\subsection{Maps onto discrete valuation rings}
In this subsection we show that we have sufficiently many surjections from $\Rsq[[u]]$ to discrete valuation rings of characteristic $0.$

\begin{lemma}
\label{interp_DVRpts_are_dense_in_Ru}
Let $(R,\mm_R)$ be a complete Noetherian local ring which is flat over $\Zp,$ and assume that $R/\mm_R$ is algebraically closed of characteristic $p$. 
Then the set of primes $\pp$ of $R[[u]]$ such that $R[[u]]/\pp$ is a characteristic $0$ discrete valuation ring uniformised by the image of $u$ is a Zariski dense subset of $\Spec R[[u]].$
\end{lemma}
\begin{proof}
Denote the set in question by in $P$ and its Zariski closure by $\overline{P}.$
We first establish that the set $P_0$ of $\pp_0 \in \Spec R$ such that $R/\pp_0$ is a characteristic $0$ ring of Krull dimension $1$ is Zariski dense in $\Spec R.$ To see this, note that $\Spec R[p^{-1}]$ is a Jacobson scheme (using \cite[Tag 01P4 and Tag 02IM]{stacks} and noting that $\Spec R[p^{-1}]$ is an open subscheme of $\Spec R \setminus \mm_R$) i.e. its closed points are Zariski dense: its closed points are precisely $P_0$. Our flatness assumption guarantees that $\Spec R[p^{-1}]$ is dense in $\Spec R.$

For any $\pp_0 \in P_0,$ the fibre $\phi^{-1}(\pp_0)$ of $\phi : \Spec R[[u]] \to \Spec R$ at $\pp_0$ is $\Spec (R/\pp_0[[u]]).$ By Weierstrass preparation (using that $R/\pp_0$ is a complete local ring), any infinite subset of $\phi^{-1}(\pp_0)$ is Zariski dense in the whole fibre.

For fixed $\pp_0,$ consider the integral closure $\kappa^+(\pp_0)$ of $R/\pp_0$ in $\kappa(\pp_0),$ this is a complete discrete valuation ring which is finite over $W(R/\mm_R)$ by Lemma \ref{contains_witt_vectors_lemma}.

Since $R/\mm_R$ is assumed to algebraically closed, $\kappa^+(\pp_0)$ is a totally ramified extension of the discrete valuation ring $W(R/\mm_R).$
Therefore any uniformiser $\varpi$ of $\kappa^+(\pp_0)$ generates it as an algebra over $W(R/\mm_R).$ It also generates $\kappa^+(\pp_0)$ over $R$ which must contain $W(R/\mm_R)$ by part (a) of Lemma \ref{contains_witt_vectors_lemma}.

For any such $\varpi,$ there is a unique extension of $R \to \kappa^+(\pp_0)$ to a ring morphism $R[[u]] \to \kappa^+(\pp_0)$ sending $u$ to $\varpi.$ This is surjective by the preceding discussion, and hence its kernel belongs to $P.$ 
Fixing $\pp_0,$ the kernels obtained in this way for uniformisers $\varpi,\varpi'$ are the same if and only if there is an $R$-algebra automorphism of $\kappa^+(\pp_0)$ sending $\varpi$ to $\varpi'.$ This automorphism group is finite by part (b) of Lemma \ref{contains_witt_vectors_lemma}.

Noting that there are infinitely many choices for $\varpi$ as $\kappa^+(\pp_0)^\times$ is infinite, this construction yields infinitely many points of $P$ on the given fibre $\phi^{-1}(\pp_0)$,  hence $\overline{P}$ contains all the fibres above $P_0.$ If $\phi$ denotes the projection $\Spec R[[u]] \to \Spec R,$ $\overline{P}$ therefore intersects all sets $U \subseteq \Spec R[[u]]$ satisfying $\phi(U) \cap P_0 \neq \emptyset$. Noting that $\phi$ is open and $P_0$ is dense shows that $\overline{P} \cap U \neq \emptyset$ if $U$ is open.
\end{proof}

\begin{lemma}
\label{contains_witt_vectors_lemma}
Let $(R,\mm_R)$ be a complete Noetherian local domain which is flat over $\Zp,$ and assume that $R/\mm_R$ is perfect of characteristic $p$. Then 
\begin{enumerate}[label=(\alph*)]
    \item $R$ contains a copy of the ring of Witt vectors $W(R/\mm_R).$
    \item If moreover $R$ is an integral domain and has Krull dimension $1$ then the integral closure of $R$ in $\Frac(R)$ is finite over $W(R/\mm_R).$
\end{enumerate}
\end{lemma}
\begin{proof}
(a) By Cohen's structure theorem \cite[Tag 032A]{stacks}, $R$ is of the form
$S[[T_1,\dots,T_m]]/I$ where $S$ is a complete discrete valuation ring with residue field $R/\mm_R$ and uniformiser $p,$ $m \in \NN$ and $I$ is an ideal. We deduce $p \not\in I$ by the $\Zp$-flatness of $R$, and so the restriction $S \to R$ is injective.

Since $R/\mm_R$ is perfect, by the theory of Witt vectors we must have $S \cong W(R/\mm_R)$ \cite[Ch. IX, §4, Proposition 6]{bourbaki_ac}.

(b) Using the finiteness of integral closure for complete Noetherian domains \cite[Theorem 4.3.4]{swansonhuneke}, and by the transitivity of finiteness, it is sufficient to prove that $R$ is finite over $W(R/\mm_R).$ Since $R$ is flat over $\Zp,$ it is also flat over $W(R/\mm_R)$ (which is uniformised by $p$) and therefore we have
$$\dim R/pR = \dim R - \dim W(R/\mm_R) = 0$$
and so $R/pR$ is an Artinian ring. Then it is also of finite length over itself and thus $R.$ In particular, it has a Jordan-H\"older decomposition over $R$ with all components isomorphic to $R/\mm_R,$ which is (trivially) finitely generated over $W(R/\mm_R).$ The finiteness of $R$ over $W(R/\mm_R)$ now follows from Nakayama's lemma for the complete local ring $W(R/\mm_R)$.
\end{proof}

\subsection{The integral Bernstein-Zelevinsky derivative and co-Whittaker representations}
In this subsection we recall some theory we need from \cite{helm2016whittaker} and \cite{eh} to construct universal co-Whittaker covers.

Let $\ddelta$ be the $n$th derivative functor defined in \cite{eh}, taking smooth $W(k)[G_\ell]$-modules to $W(k)$-modules. We recall some of its properties as follows.
\begin{prop}
\label{ddelta_props_prop}
$\ddelta$ satisfies the following:
\begin{enumerate}[label=(D\arabic*)]
\item $\ddelta$ is exact and $W(k)$-linear
\item $\ddeltap[M \otimes_{W(k)} N]=\ddelta[M] \otimes_{W(k)} N$ for all smooth $W(k)[G_\ell]$-modules $N$ and $M,$
\item $\ddelta$ admits a left adjoint.
\end{enumerate}
\end{prop}

For each $K \in \mcK$, we say that a smooth $K[G_\ell]$-module $M$ is \emph{generic} if $\ddelta[M] \neq 0.$ Note that for all $k_1,k_2 \in \mcK$ with an embedding $\iota: k_1 \to k_2$, and a generic $M$, the tensor product $M \otimes_{\iota} k_2$ is generic by (D2).

The following definition appears originally in \cite{eh} as the property \emph{essentially AIG}.
For all $K \in \mcK$, we say that $M$ belongs to $\eG(K)$ if
\begin{enumerate}[label=(eG\arabic*)]
\item $M$ is of finite length as a smooth $K[G_\ell]$-module,
\item $\soc M$ is absolutely irreducible and generic,
\item The map $\ddeltap[\soc M] \to \ddelta[M]$ induced by inclusion is an isomorphism, equivalently, $M/\soc M$ contains no generic subquotient.
\end{enumerate}
The only difference from \cite[Defintion 3.3]{helm2016whittaker} is the finite length requirement, which is in fact always satisfied by \cite[Corollary 5.5]{helm2016whittaker}.

If $M \in \eG(K)$ then its socle is irreducible, hence it matches the socle of any non-trivial $K[G_\ell]$-submodule. Hence $\eG(K)$ is closed with respect to taking submodules.

For an arbitrary $K[G_\ell]$-module, let $-^\sm$ denote its smooth vectors. By the smooth $K$-dual of a smooth $K[G_\ell]$-module, we mean the $K[G_\ell]$-module $\Hom_{K}(M,K)^\sm$ (where $\Hom_{K}(M,K)$ is equipped with the contragredient $G_\ell$-action). Since $G_\ell$ satisfies Condition 2.1.3 of \cite{eh}, $-^\sm$ is an exact functor, and so smooth duality is also exact.

We write $\eG^{\vi}(K)$ for the class of $W(k)[G_\ell]$-modules which are smooth $K$-dual to a $W(k)[G_\ell]$-module in $\eG(K).$ Dually, a quotient of a representation in $\eG^{\vi}(K)$ is also in $\eG^{\vi}(K).$

The following is the definition of co-Whittaker representations in \cite{helm2016whittaker}.
\begin{defn}
\label{coWh_defn}
Let $A$ be a Noetherian $W(k)$-algebra.
We call a smooth $A[G_\ell]$-module $M$ \emph{co-Whittaker} over $A$ if
\begin{enumerate}[label=(\arabic*)]
\item $M$ is admissible over $A$,
\item $\ddelta[M]$ is free of rank $1$ over $A$, and
\item if $\pp$ is a prime ideal of $A$ then $M \otimes \kappa(\pp) \in \eG^\vi(\kappa(\pp)).$
\end{enumerate}
\end{defn}

The paper \cite{helm2016whittaker} assigns to our fixed Galois representation $\overline{\rho}$ a block $\Rep_{W(k)}(G_\ell)_{[L,\pi]}$ of the category of smooth representations $\Rep_{W(k)}(G_\ell)$, consisting of representations all of whose subquotients have \emph{mod $p$ inertial supercuspidal support} given by the pair $[L, \pi]$. The Bernstein centre of this block (the ring of natural transformations $\mathrm{Id} \Rightarrow \mathrm{Id}$ on this full subcategory) is a $W(k)$-algebra denoted $A_{[L,\pi]},$ which acts on all objects of $\Rep_{W(k)}(G_\ell)_{[L,\pi]}$ by definition.
The connection between $\overline{\rho}$ and $A_{[L,\pi]}$ is that \cite[Conjecture 7.5]{helm2016whittaker}, proved in \cite{helmmoss} by Helm-Moss provides a map of rings
$$\mathrm{LL} : A_{[L,\pi]} \to  R_{\overline{\rho}}^\square$$
such that $A_{[L,\pi]},$ interpreted as a direct summand of the Bernstein centre, acts on $\tilde{\pi}(\rho^\square \otimes \kappa(x))$ through $\mathrm{LL}$ for all $x \in \Spec R_{\overline{\rho}}^\square.$

Moreover \cite{helm2016whittaker} shows the existence of a projective co-Whittaker $A_{[L,\pi]}[G_\ell]$-module $W_{[L,\pi]},$ which has a certain universal property \cite[Theorem 6.3]{helm2016whittaker} which we reformulate slightly as follows.

\begin{lemma}
\label{quotientofW_lemma}
Let $R$ be a Noetherian $W(k)$-algebra, and let $M \in \Rep_{R}(G_\ell)_{[L,\pi]}.$
Then $M$ is co-Whittaker over a quotient of $R$ if and only if there exists a surjection
$$W_{[L,\pi]} \otimes_{A_{[L,\pi]}} R \to M.$$ Any such $M$ is co-Whittaker over $R/\ann_{R}(\ddelta[M]).$
\end{lemma}
\begin{proof}
Assume that $M$ is co-Whittaker over some quotient $S$ or $R.$
Then \cite[Theorem 6.3]{helm2016whittaker} yields a surjection to $M$ from
$W_{[L,\pi]} \otimes_{A_{[L,\pi]}} S,$
to which $W_{[L,\pi]} \otimes_{A_{[L,\pi]}} R$ surjects by taking the quotient map on the second factor.\\
For the converse, let $W_{[L,\pi]} \otimes_{A_{[L,\pi]}} R \to M$ be a surjection. Then $M$ is admissible over $R$ by being a quotient of an admissible representation. By the exactness of $\ddelta,$ $\ddelta[M]$ is cyclic over $R$ hence isomorphic to $S:=R/\ann_{R}(\ddelta[M]).$ Finally, for any $\pp \in \Spec S$ 
we have maps $R \to S \to \kappa(\pp)$ and so a surjection $W_{[L,\pi]} \otimes_{A_{[L,\pi]}} R \otimes_R \kappa(\pp) \to M \otimes \kappa(\pp).$ $W_{[L,\pi]} \otimes_{A_{[L,\pi]}} \kappa(\pp)$ is co-Whittaker over $R$ by \cite[Theorem 6.3]{helm2016whittaker}, $M \otimes \kappa(\pp)$ is a quotient of an element of $\eG^\vi(\kappa(\pp)),$ which is closed with respect to quotients.
\end{proof}
We can immediately deduce that the class of $R[G_\ell]$-modules that are co-Whittaker over a quotient ring of $R$ is closed with respect to taking quotients, as well as the following
\begin{lemma}
\label{quotientofcowhittaker_lemma}
Let $A \to B$ be a map of Noetherian $W(k)$-algebras. If $M$ is co-Whittaker over $A$ then $M \otimes_A B$ is co-Whittaker over $B.$
\end{lemma}

The following two lemmas point out that surjections between $R[G_\ell]$-modules that are co-Whittaker over a quotient ring of $R$ have rather simple behavior.

\begin{lemma}
\label{d_princip_cyclic_lemma}
Let $R$ be a complete local $W(k)$-algebra.
Let $M,N$ be co-Whittaker representation of $G_\ell$ over quotients of $R$. Then any two surjections in $\Hom_{R[G_\ell]}(M,N)$ are multiples of each other (by elements of $R^\times.$)
\end{lemma}
\begin{proof}
Let $f,g \in \Hom_{R[G_\ell]}(M,N).$ Since $\Hom_R(\ddelta[M],\ddelta[N])$ is cyclic, we can assume without loss of generality that $\ddelta[g] = x \cdot \ddelta[f]$ for some $x \in R.$ Consider the morphism $g - x \cdot f.$ Its image must have $\ddelta = 0,$ and therefore be trivial (since $M$ has no non-trivial quotient with $\ddelta =0$). Then $g = x \cdot f$. If $f,g$ are both surjections, then $x$ cannot be contained in the maximal ideal $\mm_R$ as $N \neq \mm_R N$ since $N$ is $\mm_R$-adically separated by being admissible over $R.$
\end{proof}

\begin{lemma}
\label{d_princip_crucial_surjection_lemma}
Let $R$ be a complete local $W(k)$-algebra.
Let $f_1 : M \to M_1$ and $f_2 : M \to M_2$ be quotients of a co-Whittaker $R[G_\ell]$-module $M$. If there exists any surjection $j : M_1 \to M_2$ then the quotient map $f_2: M \to M_2$ factors as $M \xrightarrow{f_1} M_1 \xrightarrow{j} M_2.$
\end{lemma}
\begin{proof}
Note that $M_1$ and $M_2$ are co-Whittaker over some quotient of $R$ by Lemma \ref{quotientofW_lemma}. Then by Lemma \ref{d_princip_cyclic_lemma}, we have 
$$(j \circ f_1)=r \cdot f_2$$ for some $r \in R^\times.$ Hence $f_2$ factors as $j \circ (r^{-1} \cdot f_1).$
\end{proof}
In the special case when $M_1 \cong M_2$ we obtain
\begin{cor}
\label{surjkernel_indep_cor}
Let $R$ be a complete local $W(k)$-algebra, $M$ be a co-Whittaker $R[G_\ell]$-module, and let 
$N$ be a quotient of $M.$ Then $\ker f$ is independent of the choice of a quotient map $f : M \to N.$
\end{cor}

We can now prove Lemma \ref{cWh_cover_lemma}.
\begin{lemma*}
Let $R$ be a Noetherian $W(k)$-algebra, and let $(M_i)_{i \in I}$ be a collection of $R[G_\ell]$-modules, such that each is co-Whittaker over some quotient of $R$ belonging to $\mathrm{QDVR}_k$. Assume moreover that the action of the integral Bernstein center on each $M_i$ factor through the same direct summand. Then there exists a universal co-Whittaker cover of $(M_i)_{i \in I}.$
\end{lemma*}
\begin{proof}
Without loss of generality we assume that $M_i \in \Rep_{R}(G_\ell)_{[L,\pi]}$ for all $i \in I.$
By Lemma \ref{quotientofW_lemma}, we can choose surjections $\beta_i : W_{[L,\pi]} \otimes_{A_{[L,\pi]}} R \to M_i.$ By Corollary \ref{surjkernel_indep_cor}, 
$$M:=\left( W_{[L,\pi]} \otimes_{A_{[L,\pi]}} R \right) / \bigcap_{i \in I} \ker \beta_i$$
is independent of the choice of the $\beta_i$. If $M'$ is co-Whittaker over some quotient of $R$ and admits surjections to all $M_i,$ then we have maps
$$W_{[L,\pi]} \otimes_{A_{[L,\pi]}} R \to M' \to \prod_{i \in I} M_i$$
by Lemma \ref{quotientofW_lemma}. However, the image of $W_{[L,\pi]} \otimes_{A_{[L,\pi]}} R$ under the composite map must be isomorphic to $\left( W_{[L,\pi]} \otimes_{A_{[L,\pi]}} R \right) / \bigcap_{i \in I} \ker \beta'_i$ for some other choice of surjections $\beta'_i : W_{[L,\pi]} \otimes_{A_{[L,\pi]}} R \to M_i,$ and so isomorphic to $M.$ This shows the universal property for $M$.
\end{proof}

\subsection{The Langlands correspondence in families}
We will work with the $R_{\overline{\rho}}^\square[[u]][G_\ell]$-module $W^\square,$ which we define to be $$W_{[L,\pi]} \otimes_{A_{[L,\pi]}} R_{\overline{\rho}}^\square[[u]].$$ It is co-Whittaker by \cite[Theorem 6.3]{helm2016whittaker}. Moreover, it is projective over $R_{\overline{\rho}}^\square[[u]][G_\ell]$ since $W_{[L,\pi]}$ is projective over $A_{[L,\pi]}[G_\ell].$

\noindent\textbf{Notation.} We will write $A:=R_{\overline{\rho}}^\square[[u]]$ for brevity throughout the rest of the paper.
\begin{defn}
Let $P$ be the set of points $x$ of $A=\Spec R_{\overline{\rho}}^\square[[u]]$ with corresponding prime ideal $\pp_x$ such that $A/\pp_x$ is a discrete valuation ring of characteristic $0$ and the image of $u$ in the quotient is a uniformiser.
\end{defn}
Lemma \ref{interp_DVRpts_are_dense_in_Ru} shows that $P$ is Zariski dense.

For $x \in \Spec(A)$ and corresponding a (framed) deformation $\rho_x : G_E \to \GL_n(A/\pp_x)$ given by $\rho_x=\rho^\square \otimes_A A/\pp_x$, \cite{eh} assigns a representation $\tilde{\pi}(\rho_x)$ (independently of the framing) which is torsion-free and co-Whittaker \cite{helm2016whittaker}. In particular, it is admissible over $A/\pp_x$, and satisfies $\tilde{\pi}(\rho_x) \otimes_{A/\pp_x} \Frac(A/\pp_x) \cong \tilde{\pi}(\rho_x \otimes_{A/\pp_x}  \Frac(A/\pp_x)),$ and it is the unique such $(A/\pp_x)[G_\ell]$-module up to isomorphism. We will often use this uniqueness property.

If $x \in P$ then $R_{\overline{\rho}}^\square[[u]]/\pp_x = \kappa^+(x),$ in particular, any uniformiser of the discrete valuation ring $\kappa^+(x)$ lifts to $R_{\overline{\rho}}^\square[[u]].$

\begin{prop}
\label{interp_surjection_to_pirhox_prop}
For all $x \in P,$
\begin{enumerate}[label=(\alph*)]
    \item There is a surjection $W^\square \otimes_A \kappa(x) \to \tilde{\pi}(\rho^\square \otimes_A \kappa^+(x)) \otimes_{A} \kappa(x)$
    \item If $M$ is co-Whittaker over $A$ and admits a surjection $M \otimes_A \kappa(x) \to \tilde{\pi}(\rho \otimes_A \kappa^+(x)) \otimes_{A} \kappa(x)$ then there exists a surjection $M \to \tilde{\pi}(\rho \otimes_A \kappa^+(x))$ which is unique up to a factor in $\kappa^+(x)^\times.$
\end{enumerate}
\end{prop}
\begin{proof}
(a) Such a surjection is guaranteed to exist by \cite[Proposition 5.4]{helm2016whittaker}, and is unique up to a factor in $\kappa(x)^\times$ by \cite[Proposition 6.2]{helm2016whittaker}.\\
(b) Let $\beta$ be any surjective map $M \otimes_A \kappa(x) \to \tilde{\pi}(\rho^\square \otimes_A \kappa^+(x)) \otimes_{A} \kappa(x).$ 
Such a $\beta$ is guaranteed to exist by \cite[Proposition 5.4]{helm2016whittaker}, and is unique up to a factor in $\kappa(x)^\times$ by \cite[Proposition 6.2]{helm2016whittaker}. Let $\varpi$ be a lift of a uniformiser of $\kappa^+(x)$ to $A.$ By the co-Whittaker property, $\ddelta[\beta] : \ddelta[M] \to \ddeltap[ \tilde{\pi}(\rho^\square \otimes_A \kappa^+(x)) \otimes_A \kappa(x)]$ is a surjective map between the free modules $A \otimes_A \kappa(x) \to \kappa(x).$ We have $\ddeltap[M \otimes_A \kappa^+(x)] \cong \kappa^+(x)$,  so $\varpi^d \beta$ for some unique $d \in \ZZ$ is such that $\ddelta \circ \beta$ restricted to $M \otimes_A \kappa^+(x) \subseteq \kappa(x)$ surjects to $\kappa^+(x).$
Using that the cosocle of $\tilde{\pi}(\rho^\square \otimes_A \kappa^+(x))$ has $\ddelta \neq 0,$ surjectivity of $\ddelta \circ \varpi^d \beta|_{M}$ to $\kappa^+(x)$ is equivalent to the surjectivity of $\varpi^d \beta|_{M}$ to $\tilde{\pi}(\rho^\square \otimes_A \kappa^+(x))$ by the argument found at the end of the proof of \cite[Proposition 5.4]{helm2016whittaker}. We conclude that $\varpi^d \beta|_{W}$ is surjective to $\tilde{\pi}(\rho \otimes_A \kappa^+(x)).$ The decomposition $\kappa(x)^\times = \varpi^\ZZ \times \kappa^+(x)^\times$ shows uniqueness up to the claimed factor.
\end{proof}

\subsection{Definition of $\tilde{\Pi}_{\overline{\rho}}$}

\begin{defn}
Choose a family of surjections $(\beta_x)_{x \in P}$ as in Proposition \ref{interp_surjection_to_pirhox_prop}. Let
$$\tilde{\Pi}_{\overline{\rho}} := \im \left( W^\square \xrightarrow{\prod \beta_x } \prod_{x \in P}  \tilde{\pi}(\rho^\square \otimes \kappa^+(x))  \right)$$
be the image of the product of these maps considered as an $A[G_\ell]$-module.
\end{defn}
We claim that this image is well-defined up to isomorphism as an $A[G_\ell]$-module. Indeed, $$\tilde{\Pi}_{\overline{\rho}}=W^\square/\bigcap_{x \in P} \ker \beta_x$$ and $\ker \beta_x$ only depends on $x$ by Proposition \ref{interp_surjection_to_pirhox_prop}.

\begin{prop}
\label{coWh_identification_prop}
$\tilde{\Pi}_{\overline{\rho}} = \tilde{\Pi}^{\mathrm{DVR},u}(\rho^\square).$
\end{prop}
\begin{proof}
By the transitivity of universal co-Whittaker covers, $\tilde{\Pi}^{\mathrm{DVR},u}(\rho^\square)$ is the universal co-Whittaker cover of all the representations
$$\tilde{\pi}(\rho) \otimes_{\tilde{S}} S$$
where $S \in \mathrm{QDVR}_k$ is a quotient of $A$ sending $u$ to a uniformiser of $S$ and $\rho : G_{F_\ell} \to \GL_n(\tilde{S})$ is a lift of $\rho^\square \otimes S$ to the characteristic $0$ discrete valuation ring $\tilde{S} \in \mathrm{DVR}_S.$ Since $\Rsq$ is a deformation ring, for any such data there is a homomorphism $\alpha : \Rsq \to \tilde{S}$ such that $\rho\cong \rho^\square \otimes \tilde{S}.$ Lifting the image of $u$ in $S$ to a uniformiser in $\tilde{S},$ we obtain a homomorphism $\alpha' : A \to \tilde{S}$ sending $u$ to a uniformiser. These are canonically in bijection with the set $P.$ Conversely, for any $\pp \in P$ we can set $\tilde{S}=\kappa^+(\pp),$ $\rho=\rho^\square \otimes \kappa^+(\pp)$ and $S=\kappa^+(\pp)/(u^j)$ for any $j \in \NN.$\\
Therefore a co-Whittaker cover of all the $\tilde{\pi}(\rho) \otimes_{\tilde{S}} S$ is equivalent to a co-Whittaker cover of
$$\left(\tilde{\pi}(\rho^\square \otimes \kappa^+(\pp)) \right)_{\pp \in P}.$$
Now we recognise that the definition of $\tilde{\Pi}_{\overline{\rho}}$ is precisely the construction of the universal co-Whittaker cover over $A$ given in Lemma \ref{cWh_cover_lemma} for this set of representations.
\end{proof}

\begin{prop}
\label{interp_bigpi_is_coWh_prop}
$\tilde{\Pi}_{\overline{\rho}}$ is a co-Whittaker $A[G_\ell]$-module.
\end{prop}
\begin{proof}
By Lemma \ref{quotientofW_lemma}, it is sufficient to prove that $\ddelta[\tilde{\Pi}_{\overline{\rho}}]$ is a faithful module over $A.$ (It is then also free of rank $1.$)
Indeed, by the exactness of $\ddelta,$ we have 
$$\ddeltap[\tilde{\Pi}_{\overline{\rho}}] = \ddeltap[W^\square]/\ddeltap[\bigcap_{x \in P} \ker \beta_x] = \ddeltap[W^\square]/\bigcap_{x \in P} \ker \ddelta[\beta_x]=A/\bigcap_{x \in P} \pp_x=A$$
where we have also used that $\ddelta$ preserves filtered limits (by being a right adjoint), and that $P$ is Zariski dense.
\end{proof}

\subsection{Lower bounds on $\tilde{\Pi}_{\overline{\rho}}$}

\begin{lemma}
\label{Wsq_pi_injects_to_product_lemma}
Let $\Spec S \to \Spec A$ be a closed embedding of schemes, and assume that there is a set $P_{1,S} \subseteq P \cap \Spec S$ which is Zariski dense in $\Spec S$ and for each $\pp \in P_{1,S},$
$$\tilde{\pi}(\rho^\square \otimes S) \otimes_S S/\pp \cong \tilde{\pi}(\rho^\square \otimes S/\pp).$$

Then we can choose maps $W^\square \twoheadrightarrow \tilde{\pi}(\rho^\square \otimes S) \hookrightarrow \prod_{\qq \in P_{1,S}} \tilde{\pi}(\rho^\square \otimes \kappa^+(\pp))$
\end{lemma}
\begin{proof}
We can choose a surjection $W^\square \to \tilde{\pi}(\rho^\square \otimes S)$
by noting that $\tilde{\pi}(\rho^\square \otimes S)$ is co-Whittaker by its defining property, and applying \cite[Theorem 6.3]{helm2016whittaker} (recall that $S$ is canonically an $A_{[L,\pi]}$-algebra through the maps $A_{[L,\pi]} \xrightarrow{\mathrm{LL}} \Rsq \to S.)$

For the injection, note that we have maps
 $$\tilde{\pi}(\rho^\square \otimes S) \hookrightarrow \prod_{\pp \in P_{1,S}} \tilde{\pi}(\rho^\square \otimes S) \otimes_S S/\pp \xrightarrow{\sim} \prod_{\pp \in P_{1,S}} \tilde{\pi}(\rho^\square \otimes S/\pp).$$
 The first map is injective, since $P_{1,S}$ is Zariski dense and $\tilde{\pi}(\rho^\square \otimes S)$ is $S$-torsionfree by being co-Whittaker. The isomorphism follows by the assumption on $P_{1,S}.$
\end{proof}

\begin{prop}
\label{lower_bound_prop}
Let $\Spec S \to \Spec A$ be a closed embedding such that 
$$P_{1,S}:=\left\{ \pp \in P \cap \Spec S : \tilde{\pi}(\rho^\square \otimes S) \otimes_S S/\pp \cong \tilde{\pi}(\rho^\square \otimes S/\pp) \right\}$$ is Zariski dense inside $\Spec S.$ Then 
there is a surjection $\tilde{\Pi}_{\overline{\rho}} \to \tilde{\pi}(\rho^\square \otimes S).$
\end{prop}
\begin{proof}
Using Lemma \ref{Wsq_pi_injects_to_product_lemma} to choose maps, we have the following diagram:
  $$\begin{tikzcd}
 W^\square \ar{d} & ~ & \prod_{x \in P} \tilde{\pi}(\rho^\square \otimes \kappa^+(x) ) \ar[two heads]{d}\\
 \tilde{\pi}(\rho^\square \otimes S) \ar{rr} & ~ & \prod_{x \in P_{1,S}} \tilde{\pi}(\rho^\square \otimes \kappa^+(x) )
  \end{tikzcd}
 $$
 where the right hand side arrow is the natural projection.
 
 Choosing arbitrary surjections $\beta_x : W^\square \to \tilde{\pi}(\rho^\square \otimes \kappa^+(x) )$ for $x \in P \setminus P_1$ allows us to lift $\iota \circ \sigma$ to a surjective map $W^\square \to \prod_{x \in P} \tilde{\pi}(\rho^\square \otimes \kappa^+(x) )$ that makes the diagram commute. It has image isomorphic to $\tilde{\Pi}_{\overline{\rho}}$ by definition of the latter.
 The claim now follows from the injectivity of $\iota$ and the commutativity of the diagram by comparing the images of $W^\square$.
 \end{proof}

\section{The deformation ring and minimal lifts}
\subsection{Setup on the Galois side}
We introduce some notation for, and recall some facts about the Galois theory of $F_\ell$.
Its absolute Galois group $G_{F_\ell}$ is a pro-finite group.
We have the inertia subgroup $\Ifl \le G_{F_\ell}$ which is a normal subgroup with quotient generated by the class of a Frobenius element $\phi.$ $\Ifl$ has a topologically cyclic pro-finite quotient isomorphic to $(\widehat{\ZZ},+),$ with corresponding kernel called the wild inertia subgroup. Therefore there is a unique subgroup $\Pfl \le \Ifl$ with quotient isomorphic to $(\Zp,+).$ Consequently, $\Pfl$ is also a normal subgroup of $G_{F_\ell}$. We call representations factoring through $G_{F_\ell}/\Pfl$ $p$-tame.

Choose topological generators $(\phi,\sigma)$ for the topological group $G_{F_\ell}/\Pfl$ satisfying $\phi \sigma \phi^{-1}=\sigma^q$ where $q$ is the cardinality of the residue field of $F_\ell.$ (Note that a disadvantage of our notational choices is that here, $q$ is a power of $\ell,$ not of $p.$)

\begin{defn}
For an irreducible representation $\overline{\tau}$ of the group $\Pfl$ over $k,$ we define $\Got$ to be the stabiliser subgroup
$$\left\{g \in G_{F_\ell} : {}^g(\overline{\tau}) \cong \overline{\tau} \right\}$$
where ${}^g(\overline{\tau})$ denotes the conjugate representation $h \mapsto \overline{\tau}(g h g^{-1})$ of $\Pfl$ over $k.$
\end{defn}

\begin{lemma}[{\cite[Lemma 2.4.11]{CHT}}]
$\overline{\tau}$ lifts uniquely to a representation $\tau$ of $\Pfl$ with coefficients in $W(k)$ Moreover, this $\tau$ extends (non-uniquely) to a representation of $\Got$ with coefficient in $W(k).$
\end{lemma}

\begin{defn}
We choose precisely one $\overline{\tau}$ from each $G_{F_\ell}$-conjugacy class of irreducible representations of $\Pfl$ over $k$ that appear as an irreducible $k[\Pfl]$-subquotient of $\overline{\rho}$, and collect them in a set $\overline{T}.$ According to the Lemma, we choose an extension to $\Got$ of lift $\tau$ for each $\overline{\tau} \in \overline{T}$ and denote by $T$ the set of our chosen extensions $\tau$ to $\Got.$ Note that $T$ is a finite set of cardinality at most $\dim_k \overline{\rho}.$
\end{defn}

For any $\tau \in T$ and a $W(k)[G_{F_\ell}]$-module $M,$ we have an action of $\Got/\Pfl$ on $\Hom_{\Pfl}(\tau,M)$. By \cite[Lemma 2.4.12]{CHT}, we have
$$\rho \cong \bigoplus_{\tau \in T} \Ind_{G_{\overline{\tau}}}^{G_{F_\ell}}
\left( \Hom_{W(k)[\Pfl]}(\tau, \rho) \otimes_{W(k)} \tau \right)$$
for all $W(k)[G_\ell]$-modules $\rho$ such that all simple $k[\Pfl]$-subquotients of $\rho$ are contained in $\overline{T}.$

\subsection{Ramification and induction}

In this subsection, let $K$ be a characteristic $0$ field of fractions of a complete Noetherian local integral $W(k)$-algebra (so that \cite[Proposition 4.1.6]{eh} applies). Note that $\kappa(\pp)$ satisfies the requirements for any $\pp \in \Spec \Rsq[p^{-1}].$ Fix some $\tau \in T.$

\begin{lemma}
\label{restriction_depends_on_restriction_lemma}
For each $\tau \in T,$ there is an open subgroup $U_{\tau} \le I_{F_\ell}$ such that 
$$\Res^G_{U} \Ind_{\Got}^{G_{F_\ell}} \psi \otimes_{W(k)} \tau$$
depends only on $\Res^{\Got}_{U \cap \Got} \psi$ up to isomorphism for any ($p$-tame) representation $\psi : \Got/\Pfl \to \GL_r(K)$ and open subgroup $U \subseteq U_\tau.$
\end{lemma}
\begin{proof}
Note that $\Got$ is the absolute Galois group of some finite extension of $F_\ell,$ so \cite[Proposition 4.1.6]{eh} applies to it as well. In particular there is an open subgroup $U \le \Ifl$ and a nilpotent matrix $N_\tau \in M_{r \times r}(K)$ such that $\tau(u)=\exp(t_p(u) N_\tau)$ for all $u \in U$ (which implies the same claim for all open subgroups of $U$.)\\
Now let $U$ be such an open subgroup, and let $U_0$ be the projection of $U$ to $G_{F_\ell}/\Pfl.$ Since $\Ifl/\Pfl$ is a topologically cyclic pro-$p$ group, $U_0$ is the unique subgroup of $I_{F_\ell}/\Pfl$ of a given index, and so $U_0$ is normal in $G_{F_\ell}/\Pfl.$ Then $\Res_{U}^{\Got} \left( \psi \otimes_{W(k)} \tau \right)$ factors through $U_0,$ and so we have ${}^g(\psi \otimes_{W(k)} \tau|_{U})= (\psi \otimes_{W(k)} \tau|_{U})^{\chi(g)}$ where we use the notation ${}^g (\psi \otimes_{W(k)} \tau)=(u \mapsto (\psi \otimes_{W(k)} \tau)(g u g^{-1})),$ and $\chi : G_{F_\ell} \to \Zp^\times$ is the unramified character sending the Frobenius to $q.$
Using the double coset formula for the restriction of an induced representation,
\begin{align*}
    \Res_{U}^{G_{F_\ell}} \Ind_{\Got}^{G_{F_\ell}} \left( \psi \otimes_{W(k)} \tau \right)=&
    \bigoplus_{U g \Got} \Ind_{U \cap g \Got g^{-1}}^{U} \quad   {}^g(\psi \otimes_{W(k)} \tau)=\\
    &\bigoplus_{U g \Got} \Ind_{U \cap g \Got g^{-1}}^{U}    \quad  (\psi \otimes_{W(k)} \tau)^{\chi(g)}.
\end{align*}
By inspection, the last line depends only on the restriction of $\psi \otimes_{W(k)} \tau$ to $U,$ and so setting our $U$ as $U_{\tau}$ satisfies the claim.
\end{proof}

\subsection{Partitions and unions of components}
Continuing from the previous subsection, let $K$ be a characteristic $0$ field of fractions of a complete Noetherian local integral $W(k)$-algebra (so that \cite[Proposition 4.1.6]{eh} applies).

The following definition and its properties are well-known.
\begin{defn}
If $N,N' \in M_{n \times n}(K)$ are nilpotent matrices, we say $N \le N'$ if and only if
$$\rk(N^i) \le \rk(N'^i)$$ for all $i = 1 \dots n-1.$ This relation only depends on $N$ and $N'$ up to conjugation over $\overline{K},$ and defines a partial order on $\overline{K}$-conjugacy classes of nilpotent matrices. 
\end{defn}

Let $\mu$ be a partition of $m \in \NN$ i.e. a non-increasing sequence of positive integers $\mu_1,\dots,\mu_r$ summing to $m.$ We extend $\mu$ into an infinite sequence by setting $\mu_i=0$ if $i>r$ for notational convenience.
Our reference for facts about partitions is \cite{brylawski}.

For each partition $\mu$ and field $K,$ there is a unique nilpotent $n \times n$ matrix $N_\mu(K)$ in Jordan normal form such that $\mu_1,\dots,\mu_r$ are the length of its Jordan blocks from left to right. We say that $\mu$ dominates a nilpotent matrix $N \in M_{m \times m}(K)$ and write $N \le \mu$ if $N \le N_\mu(K).$

\begin{defn}[Dominance order on partitions]
If $\mu,\nu$ are partitions of an integer $m \in \NN,$ we say that $\mu \le \nu$ if and only if $$\sum_{i=1}^k \mu_i \le \sum_{i=1}^k \nu_i$$ for any integer $k \ge 1.$ We write $\Part(m)$ for this partially ordered set.
\end{defn}

Note that any nilpotent $m \times m$ matrix over $K$ is conjugate to some $N_\mu(K)$ over $\overline{K},$ (in fact, even over $K$) so $\mu \mapsto [N_\mu(\overline{K})]$ is an order-preserving bijection between partitions of $m$ and conjugacy classes of nilpotent matrices over $\overline{K}.$

\begin{defn}
If $N$ is any nilpotent $m \times m$ matrix (resp. a conjugacy class of matrices) over $K$, 
there is a unique partition $\nu$ of $m$ such that $N$ is conjugate to (resp. contains) $N_{\nu}(K)$ over $K.$ We write $\nu(N)$ for this unique partition.
\end{defn}

There is an order-reversing involution known as \emph{conjugation} on partitions of $m$ given by $\mu \mapsto \mu^\top$ such that
$$\mu^\top_i=|\{j \in \{1,\dots, r\} :  \mu_j \ge i \}|.$$

The following is well-known.
\begin{lemma}
If $\mu,\nu$ are partitions of $m,$ then $$N_\mu(K) \le N_\nu(K) \Leftrightarrow \mu \le \nu.$$
\end{lemma}
\begin{proof}
Since $\rk N=m-\dim_K \ker N$ for any $m \times m$ matrix $N,$ 
$N_\mu(K) \le N_\nu(K)$ is equivalent to $\dim_K \ker N_\mu(K)^i \ge \dim_K \ker N_\nu(K)^i$ for all $i \in \NN.$ However, $(\dim_K \ker N_\mu(K)^i)-(\dim_K \ker N_\mu(K)^{i-1})$ is the number of blocks of $N_\mu(K)$ of length at least $i$ (and similarly for $\nu$), hence the claim is equivalent to $\mu^{\top} \ge \nu^{\top}$ where $-^\top$ denotes the conjugate partition. Since $-^\top$ is an order-reversing involution on partitions of $m$, this implies the claim.
\end{proof}

The following is also well-known.
\begin{prop}
For a partition $\mu$ of $m$, the $N \le \mu$ condition cuts out a Zariski closed subset of the moduli space of $m \times m$ matrices over $K$, which is the Zariski closure of the closed points corresponding to matrices conjugate to $N_\mu(K)$ over $\overline{K}.$
\end{prop}

Using \cite[Proposition 4.1.6]{eh}, the following is well-defined.
\begin{defn}
Let $U$ be any open subgroup of $G_{F_\ell},$ and let $\rho : U \to \GL_m(K)$ be a continuous Galois representation, which extends to a continuous representation of $G_{F_\ell}.$
If we let $(\rho',N)$ denote the Weil-Deligne representation attached to $\rho,$ we call the matrix $N$ \emph{the monodromy of} $\rho.$ We write $\nu(\rho)$ for the partition $\nu(N)$ of $m.$
The conjugacy class of the monodromy of $\rho$ and so $\nu(\rho)$ depend on $\rho$ only up to isomorphism, and do not change when restricting $\rho$ to an open subgroup $U' \subseteq U.$
\end{defn}

\begin{lemma}
\label{monotonicity_lemma}
For a representation $\psi : \Got/\Pfl \to \GL_r(K)$ with attached Weil-Deligne representation $(\psi',N_\psi),$ let $N_{\psi,\tau}$ be such that $(\rho_{\psi,\tau},N_{\psi,\tau})$ is the Weil-Deligne representation attached to $\Ind_{\Got}^{G_{F_\ell}} \psi \otimes_{W(k)} \tau.$ Then the conjugacy class of $N_{\psi,\tau}$ over $\overline{K}$ depends only on the conjugacy class of $N_\psi.$ Equivalently,
$$\nu \left(   \Ind_{\Got}^{G_{F_\ell}} \psi \otimes_{W(k)} \tau  \right)$$
depends only $\nu(\psi).$
Moreover, this dependence is monotonic in the partial order $\le.$
\end{lemma}
\begin{proof}
Take $U_{\tau}$ as in Lemma \ref{restriction_depends_on_restriction_lemma}. On one hand, $\psi$ is $p$-tame and so $N_\psi$ determines the entire representation $\psi|_{U_{\tau}}$ (since it determine $\psi$ on the whole of $I_{F_\ell}$). On the other hand, the restriction to $U_{\tau}$ is sufficient to determine the monodromy part of the Weil-Deligne representation.\\
To prove monotonicity, first note that $(\psi \otimes_{W(k)} \tau)^{x}$ for $x \in \ZZ^\times_p$ has the same monodromy as $\psi \otimes_{W(k)} \tau,$ so using the double coset formula in Lemma \ref{restriction_depends_on_restriction_lemma}, it is sufficient to show that the monodromy of $\psi \otimes_{W(k)} \tau$ depends monotonically on $N_\psi.$\\
From the theory of Weil-Deligne representation we know that the monodromy attached to $\psi \otimes_{W(k)} \tau$ is
$N_\psi \otimes 1 + 1 \otimes N_\tau \in \End_K(\psi \otimes_{W(k)} \tau)$ where $N_\tau$ is the monodromy attached to $K \otimes_{W(k)} \tau.$ Fix a basis for both $\psi$ and $\tau$. By conjugation over $\overline{K},$ we may assume that $N_\psi$ and $N_\tau$ are both in Jordan normal form. For $a \in \NN,$ write $V_a$ (resp. $W_a$) for the space of matrices $M$ in $\End_K(\psi)$ (resp. $\End_K(\tau)$) such that all coefficients apart from $M_{0,a},M_{1,a+1},\dots,M_{\dim \psi-a,\dim \psi}$ (resp. $M_{0,a},\dots,M_{\dim \tau - a,\dim \tau}$) are $0$. Then the sum of subspaces
$$\sum_{a,b \in \NN} V_a \otimes W_b$$
is a direct sum, and it is also naturally a direct summand of $\End_K(\psi \otimes_{W(k)} \tau).$
Then for all $m \in \NN,$ the matrix
$$(N_\psi \otimes 1 + 1 \otimes N_\tau)^m=\sum_{i=0}^m \binom{m}{i}  (N_\psi \otimes 1)^i (1 \otimes N_\tau)^{m-i}=\sum_{i=0}^m \binom{m}{i}  N_\psi^i \otimes N_\tau^{m-i}$$
has its $i$th summand lying in $V_i \otimes W_{m-i}$ and therefore its rank is the sum of the ranks of the matrices in the sum. We ignore the factors $\binom{m}{i}$ (as $\mathrm{char} K=0$) and we compute
$$\rk (N_\psi^i \otimes N_\tau^{m-i})=(\rk N_\psi^i)(\rk N_\tau^{m-i}),$$
which yields
$$\rk (N_\psi \otimes 1 + 1 \otimes N_\tau)^m = \sum_{i=0}^m (\rk N_\psi^i)(\rk N_\tau^{m-i})$$
which depends monotonically on $\rk N_\psi^i.$\\
\end{proof}

For $\tau \in T,$ let $e_\tau$ be the least positive integer such that $\sigma^{e_\tau} \in \Got.$

\begin{prop}
\label{monotonic_dependence_prop}
For representations $\rho : G_{F_\ell} \to \GL_n(K),$ the conjugacy class of the monodromy part of the attached Weil-Deligne representation 
\begin{enumerate}[label=(\alph*)]
    \item depends only on the collection of the conjugacy classes of the matrices $$( \Hom_{\Pfl}(\tau,\rho)(\sigma^{e_\tau})  )_{\tau \in T}.$$
    \item depends only on the monodromy attached to $(\Hom_{\Pfl}(\tau,\rho))_{\tau \in T},$
    \item the dependence in (b) is monotonic for the  product partial order on collections of conjugacy classes. More precisely, there exists a monotonic function
    $$F : \prod_{\tau \in T} \Part(r_\tau) \to \Part(n)$$
    such that
    $$F \left( \nu \left(  \Hom_{\Pfl}(\tau,\rho) \right)_{\tau \in T} \right) = \nu \left( \rho \right).$$
\end{enumerate}
\end{prop}
\begin{proof}
Consider the subgroup
$U=\bigcap_{\tau \in T} U_\tau$
where we take $U_\tau$ satisfying Lemma \ref{restriction_depends_on_restriction_lemma}. It is open in $I_{F_\ell}$ as $T$ is finite, and therefore its action determines the monodromy.
We have
$$\rho=\bigoplus_{\tau \in T} \Ind_{\Got}^{G_{F_\ell}} \psi \otimes_{W(k)} \tau.$$
Note that $\Hom_{\Pfl}(\tau,\rho)(\sigma^{e_\tau})$ uniquely determines the inertia part of $\Hom_{\Pfl}(\tau,\rho)$ which is $p$-tame.
Now Lemma \ref{monotonicity_lemma} proves part (a).\\
For (b) and (c), first note that taking direct sums of matrices preserves the ordering relation on nilpotent matrices, and we can use Lemma \ref{monotonicity_lemma} for each term in the sum.
\end{proof}

\subsection{Pseudo-framed deformations}
We will need the notion of \emph{pseudo-framing} from \cite{helm_curtishom}.
\begin{defn}[{\hspace{1sp}\cite[Definition 8.2]{helm_curtishom}}]
For any $W(k)$-algebra $A$, a pseudo-framed representation $\rho : G_{F_\ell} \to \GL_n(A)$ is a representation $\rho$ together with a basis $(b_{\tau,1},\dots,b_{\tau,r_\tau})$ of the free $A$-module $\Hom_{\Pfl}(\tau,\rho)$ (of rank $r_\tau$) for each $\tau \in T.$ We equip $\overline{\rho}$ with a fixed pseudo-framing, and say that a \emph{pseudo-framed} deformation of $\overline{\rho}$ over $A$ consists of pseudo-framed representation $\rho$ with an isomorphism $\rho \otimes_{A} k \xrightarrow{\sim} \overline{\rho}$ mapping the pseudo-framing basis of $\rho$ to our chosen pseudo-framing basis for $\rho.$
\end{defn}

Let $\overline{\rho}_{\overline{\tau}} : \Got \to \GL_n(k)$ be the representation defined by $\Hom_{\Pfl}(\tau,\overline{\rho}).$\\
Consider the ring 
$$\Rd := {\widehat{\bigotimes}}_{\tau \in T} 
R^\square_{{\overline{\rho}}_{\overline{\tau}} }$$
which parametrises pseudo-framed deformations of $\overline{\rho},$ and carries the universal pseudo-framed deformation $$\rho^\diamond=\bigoplus_{\tau \in T}\Ind_{\Got}^{G_{F_\ell}} (\rho_{\overline{\tau}}^\square \otimes_{W(k)} \tau)$$ (see \cite[Definition 8.2]{helm_curtishom} and the subsequent discussion).

Moreover, each ring $R^\square_{{\overline{\rho}}_{\overline{\tau}}}$ is isomorphic to the completion of some $R_{q_\tau,r_\tau}$ by the $p$-tame special case of \cite[Proposition 9.2]{helm_curtishom},
where $\Spec R_{q_\tau,r_\tau}$ is a moduli space for pairs of invertible $r_\tau \times r_\tau$ matrices $(\Phi,\Sigma)$ such that $\Phi \Sigma \Phi^{-1}=\Sigma^{q_\tau}.$ It carries universal matrices $\Phi$ and $\Sigma,$ and for some $a_\tau$ depending only on $q_\tau$ and $r_\tau,$ the latter satisfies $$\Sigma^{p^{a_{{\tau}}}}=1$$ by \cite[Proposition 6.2]{helm_curtishom}. 

We have maps
$$\Spec R^\square_{{\overline{\rho}}_{\overline{\tau}}} \to \Spec R_{q_\tau,r_\tau} \xrightarrow{\pr_{\Sigma}} \Aa^{r_\tau^2}$$
where the last map corresponds to forgetting $\Phi$ and mapping $\Sigma$ to its matrix coefficients.

\textbf{Notation.} For a point $x \in \Spec R_{q_\tau,r_\tau}$, let $\Sigma_x \in M_{r_\tau \times r_\tau}(\kappa(x))$ be the matrix corresponding to $\pr_{\Sigma}(x)$.

\begin{prop}
\label{Rqr_partition_prop}
Let $\mu$ be any partition of $r_\tau.$ Then 
$$\left\{ x \in \Spec R_{q_\tau,r_\tau} : \nu(\Sigma_x^{a_\tau} - 1) \le \mu \right\}$$
is a union of irreducible components of $\Spec R_{q_\tau,r_\tau}.$
\end{prop}
\begin{proof}
First, note that the set in question is Zariski closed since it is the preimage of the closed subset 
of $\Aa^{r_\tau^2}.$ On the other hand, whether $x$ belongs to the set only depends on the conjugacy class of $\Sigma_x \in M_{r_\tau \times r_\tau}(\kappa(x)).$ By these two facts,
$$\left\{ x \in \Spec R_{q_\tau,r_\tau} : \nu(\Sigma_x^{a_\tau} - 1) \le \mu \right\}=\bigcup_{x \in \Spec R_{q_\tau,r_\tau}} \bigcup_{c \in C(x)} \overline{\pr_{\Sigma}^{-1}(U_c)}$$
where $C(x)$ is a certain (possibly empty) set of conjugacy classes of $r_\tau \times r_\tau$ matrices over $\kappa(x)$, and $U_c$ is the locally closed subset of $\Aa^{r_\tau^2} \times_{W(k)} \Spec \kappa(x)$ cut out by $c.$ The third paragraph of the proof of \cite[Proposition 6.2]{helm_curtishom} states precisely that the Zariski closures of the preimages of these sets $U_c$ are irreducible components.
\end{proof}

\begin{prop}
\label{Ztaumu_prop}
Let $\mu$ be any partition of $r_\tau,$  and define
$$Z_{\tau,\mu}=\left\{ x \in \Spec R^\square_{{\overline{\rho}}_{\overline{\tau}}}[p^{-1}] : \nu \left( \rho^\square_{{\overline{\rho}}_{\overline{\tau}}} \otimes_{R^\square_{{\overline{\rho}}_{\overline{\tau}}}} \kappa(x) \right) \le \mu \right\}.$$
Then $Z_{\tau,\mu}$ is a union of irreducible components of $\Spec R^\square_{{\overline{\rho}}_{\overline{\tau}}}[p^{-1}]$
\end{prop}
\begin{proof}
The map $i : \Spec R^\square_{{\overline{\rho}}_{\overline{\tau}}} \to \Spec R_{q_\tau,r_\tau}$ corresponds to a completion of Noetherian rings, and is therefore flat. Thus the preimage of an irreducible component under $i$ is a union of irreducible components. Together with Proposition \ref{Rqr_partition_prop}, it is now sufficient to show that 
$$Z_{\tau,\mu}=i^{-1} \left( \left\{ x \in \Spec R_{q_\tau,r_\tau}[p^{-1}] : \nu(\Sigma_x^{a_\tau} - 1) \le \mu \right\} \right).$$
We have
$$\nu \left( \rho^\square_{{\overline{\rho}}_{\overline{\tau}}} \otimes_{R^\square_{{\overline{\rho}}_{\overline{\tau}}}} \kappa(x) \right)=\nu(\Sigma_x^{a_\tau} - 1)$$
since the monodromy of $\rho^\square_{{\overline{\rho}}_{\overline{\tau}}} \otimes_{R^\square_{{\overline{\rho}}_{\overline{\tau}}}} \kappa(x)$ is, up to a non-zero factor, the matrix logarithm of the image of the group element $\sigma^{e_\tau a_\tau}$ under the representation, which is
$\log(\Sigma_{i(x)}^{a_\tau}).$ As we can easily verify after conjugation to Jordan normal form over an algebraic closure, we have $\nu(\log M)=\nu(M-1)$ for any unipotent matrix $M.$
Therefore we have $\nu \left( \rho^\square_{{\overline{\rho}}_{\overline{\tau}}} \otimes_{R^\square_{{\overline{\rho}}_{\overline{\tau}}}} \kappa(x) \right) = \nu( \Sigma_{i(x)}^{a_\tau} -1 ),$
showing the claim.
\end{proof}

For any given partition $\mu$ of $n$, let 
$$Z_\mu:=\left\{ x \in \Spec \Rsq[p^{-1}] : \nu \left(  \rho^\square \otimes \kappa(x) \right) \le \mu \right\}.$$
and
$$Z'_\mu:=\left\{ x \in \Spec \Rd[p^{-1}] : \nu \left(  \rho^\diamond \otimes_{\Rd} \kappa(x) \right) \le \mu \right\}.$$

By choosing a pseudo-framing for $\rho^\square,$ we obtain a map
$\Psi: \Spec \Rsq \to \Spec \Rd.$
$\Psi$ establishes a bijection between irreducible components by the discussion on action of "change of frame" formal groups after \cite[Definition 8.2]{helm_curtishom} (we use bijections between the irreducible components of $\Spec R$ and  $\Spec R[[X_1,\dots,X_r]]$ for $R=\Rsq$ and $R=\Rd$).

\begin{lemma}
\label{psi_zmu_lemma}
$Z_\mu=\Psi^{-1}(Z'_\mu).$
\end{lemma}
\begin{proof}
It is sufficient to note that the partition $\nu$ assigned to a framed and/or pseudo-framed deformation of $\overline{\rho}$ is independent both of framing and pseudo-framing.
\end{proof}

\begin{prop}
$Z_\mu$ is a union irreducible components of $\Spec \Rsq[p^{-1}].$
\end{prop}
\begin{proof}
By Lemma \ref{psi_zmu_lemma} and the preceding discussion on $\Psi$, it is sufficient to prove that
$$Z'_\mu=\left\{ \pp \in \Spec \Rd[p^{-1}] : \nu \left( \rho^\diamond \otimes_{\Rd} \kappa(\pp) \right) \le \mu \right\}$$ is a union of irreducible components.

For $\tau \in T,$ let $\pr_\tau : \Spec \Rd \to \Spec R^\square_{{\overline{\rho}}_{\overline{\tau}}}$ denote the projection map induced by the completed tensor product description of $\Rd.$
Using Proposition \ref{monotonic_dependence_prop}, we have
$$Z'_\mu = \bigcup_{\boldsymbol{\alpha}} \bigcap_{\tau \in T} \pr_\tau^{-1} \left(  Z_{\tau,\alpha_\tau} \right)$$
where $\boldsymbol{\alpha}=(\alpha_\tau)_{\tau \in T} \in \prod_{\tau \in T} \Part(r_\tau)$ is such that $F(\boldsymbol{\alpha}) \le \mu$ in the notation of Proposition \ref{monotonic_dependence_prop}.
Each $Z_{\tau,\alpha_\tau}$ is a union of irreducible components, and the intersection distributes over these to write $Z'_\mu$ as the union of sets of the form $\bigcap_{\tau \in T} \pr_\tau^{-1} (C_\tau)$ where each $C_\tau$ is an irreducible component of $\Spec R^\square_{{\overline{\rho}}_{\overline{\tau}}}.$ Each of these is an irreducible component of $\Spec \Rd[p^{-1}]$ by its definition as a completed tensor product.
\end{proof}

\begin{prop}
\label{minimal_partition_prop}
Let $x$ be \emph{any} point of $\Spec \Rsq$ (i.e. we allow $\kappa(x)$ of characteristic $p$). Then 
the set
$$\{ \mu \in \mathrm{Part}(n) : x \in \overline{Z_\mu} \}$$
contains a unique element which is minimal with respect to the dominance order, which we call $\mu_x.$ \end{prop}
\begin{proof}
If $\mu$ is the partition with $\mu_1=n,$ then $Z_\mu=\Spec \Rsq[p^{-1}]$ and (by flatness over $W(k)$) also $\overline{Z_\mu}=\Spec \Rsq.$ Therefore the set of partitions in question is non-empty. That it has a minimal element follows from $\mathrm{Part}(n)$ forming a \emph{lattice} in the order-theoretic sense of the term \cite{brylawski}.
\end{proof}

The following is a reformulation of \cite[Definition 4.5.9]{eh} in terms of partitions.
\begin{defn}
Let $R$ be a complete, Noetherian, local and flat $W(k)$-algebra, and $\rho : G_{F_\ell} \to \GL_n(R)$ be a continuous Galois representation. If $\qq \le \pp \le R$ are prime ideals, we say that $\rho \otimes_R R/\qq$ is a minimal lift of $\rho \otimes_R R/\pp$ if
$$\nu \left( \rho \otimes_R R/\pp \right) = \nu \left( \rho \otimes_R R/\qq \right).$$
\end{defn}

\begin{thm}
\label{comp_finding_a_min_prime}
Let $\pp \in \Spec \Rsq[p^{-1}]$.
There is a minimal prime $\mfa \in \Spec \Rsq$ contained in $\pp$ such that
    $\rho^\square \otimes \Rsq/\mfa$ is a minimal lift of $\rho^\square \otimes \Rsq/\pp \Rsq.$
\end{thm}
\begin{proof}
Let $\mu$ be the partition corresponding to the monodromy of $\rho^\square \otimes \Rsq/\pp \Rsq.$
We have $\pp \in Z_{\mu},$ which is a union of irreducible components, one of which must contain $\pp.$ Let $\mfa$ be the minimal prime corresponding to this component. The partition corresponding to the monodromy of $\rho^\square \otimes \Rsq/\mfa$ is, on one hand, not greater than $\mu$ in the dominance order because $\mfa \in Z_\mu$ by assumption. On the other hand, it is not less than $\mu$ since specialization cannot increase the rank of a matrix.
\end{proof}

\section{Comparison with $\tilde{\pi}(\rho^\square)$}
In this section, we relate $\tilde{\pi}(\rho^\square)$ to $\tilde{\Pi}_{\overline{\rho}}.$ 
Note that these are defined over $\Rsq$ and $A=\Rsq[[u]]$ respectively, which causes some technical issues that we have to deal with.

\subsection{Surjection from $\tilde{\pi}(\rho^\square)$}
\begin{thm}
\label{comp_surj_iso_existence_thm}
Let $\pp \in \Spec \Rsq[p^{-1}]$. There is a surjection
$$\tilde{\pi}(\rho^\square) \otimes \kappa(\pp) \to \tilde{\pi}(\rho^\square \otimes \kappa(\pp)).$$ Moreover, if either \cite[Conjecture 6.2.7]{eh} holds (which is always true if $n=2$ or $n=3$) or there is a unique minimal prime contained in $\pp,$ then this is an isomorphism.
\end{thm}
\begin{proof}
By Theorem \ref{comp_finding_a_min_prime}, the minimal lift condition of \cite[Proposition 6.2.10]{eh} holds and we obtain obtaining a surjection. If $\pp$ \cite[Conjecture 6.2.7]{eh} holds, then the second part of \cite[Theorem 6.2.6]{eh} gives an isomorphism. Alternatively, if $\pp$ is contained in a unique irreducible component, we can use \cite[Theorem 6.2.5]{eh} directly.
\end{proof}

\begin{lemma}
\label{comp_surj_iso_implication_lemma}
Let $\pp' \in \Spec A[p^{-1}]$ with $\pp=\pp' \cap \Rsq.$
A surjection (resp. isomorphism)
$$\tilde{\pi}(\rho^\square) \otimes \Rsq/\pp \to \tilde{\pi}(\rho^\square \otimes \Rsq/\pp)$$ implies the existence of a surjection (resp. isomorphism)
$$\tilde{\pi}(\rho^\square) \otimes \kappa(\pp') \to \tilde{\pi}(\Rsq \otimes \kappa(\pp')).$$
\end{lemma}
\begin{proof}
We apply $- \otimes A$ to both sides and recognise that $(\Rsq/\pp) \otimes A=A/\pp A$ is canonically identified with $(\Rsq/\pp)[[u]]$ to obtain a map
$$ \tilde{\pi}(\rho^\square) \otimes A/\pp A \to \tilde{\pi}(\rho^\square \otimes \Rsq/\pp) \otimes_{\Rsq/\pp} (\Rsq/\pp)[[u]],$$
still a surjection (resp. isomorphism) by right exactness.\\
The right hand side is $(\Rsq/\pp)[[u]]$-torsionfree since $\tilde{\pi}(\rho^\square \otimes \Rsq/\pp)$ is $(\Rsq/\pp)$-torsionfree by definition, and also co-Whittaker. Now localising at the prime $\pp',$ we obtain a map
\begin{equation}
\label{eq:2}
\tilde{\pi}(\rho^\square) \otimes (A/\pp)_{\pp'} \to \tilde{\pi}(\rho^\square \otimes \Rsq/\pp) \otimes (\Rsq/\pp)[[u]]_{\pp'},
\end{equation}

a surjection (resp. isomorphism) by right exactness. Again, the right hand side remains torsion-free and co-Whittaker. Note that $(\Rsq/\pp)[[u]]_{\pp'}$ is a discrete valuation ring with residue field $\kappa(\pp').$
We can now apply the uniqueness of $$\tilde{\pi}(\rho^\square \otimes (\Rsq/\pp)[[u]]_{\pp'})$$ to show that it is isomorphic to the right hand side since they agree at the generic fibre. 

Moreover, $\rho^\square \otimes (\Rsq/\pp)[[u]]_{\pp'}$ is a minimal lift of $\rho^\square \otimes \kappa(\pp')$ since the nilpotent matrix $N$ attached to $\rho^\square \otimes (\Rsq/\pp)[[u]]_{\pp'}$ has entries lying in the subring $\Rsq/\pp.$ Therefore, by \cite[Theorem 6.2.5]{eh}, there is an isomorphism $\tilde{\pi}(\rho^\square \otimes (\Rsq/\pp)[[u]]_{\pp'}) \otimes_{(\Rsq/\pp)[[u]]_{\pp'}} \kappa(\pp') \cong \tilde{\pi}(\rho^\square \otimes \kappa(\pp') ).$
Passing to the special fibre in \eqref{eq:2} (which is right exact) and using these isomorphisms on the right hand side, we obtain a surjection (resp. isomorphism)
\begin{align*}
\tilde{\pi}(\rho^\square) \otimes \kappa(\pp') \to& \left( \tilde{\pi}(\rho^\square \otimes \Rsq/\pp) \otimes (\Rsq/\pp)[[u]]_{\pp'} \right) \otimes_{A/\pp A} \kappa(\pp')\\ &\cong \tilde{\pi}(\rho^\square \otimes (\Rsq/\pp)[[u]]_{\pp'}) \otimes_{A/\pp A} \kappa(\pp')\\ &\cong \tilde{\pi}(\Rsq \otimes \kappa(\pp'))
\end{align*}
 as desired.
\end{proof}

\begin{defn}
Let $P_1$ be the subset of $P$ consisting of points lying on only one irreducible component.
\end{defn}

\begin{prop}
\label{P1_is_dense_prop}
$P_1$ is Zariski dense in $\Spec A.$ Moreover,
$$P_1 \subseteq \left\{ \pp \in P : \tilde{\pi}(\rho^\square \otimes A) \otimes_A A/\pp \cong \tilde{\pi}(\rho^\square \otimes A/\pp) \right\}.$$
\end{prop}
\begin{proof}
It is sufficient to prove that $P_1$ is Zariski dense in $A[p^{-1}]$ since $A$ is flat over $W(k).$\\
$\Rsq[p^{-1}]$ is generically smooth by being a scheme over the spectrum of the perfect field $W(k)[p^{-1}].$ Therefore its set of points lying on a single irreducible component contains an open dense subset $U_0$ (viz. the subscheme of smooth points). 
Write $\phi : \Spec A \to \Spec \Rsq$ for the natural projection morphism.
For any prime $\pp_0 \in U_0,$ we have an isomorphism $\tilde{\pi}(\rho^\square) \otimes \Rsq/\pp_0 \cong \tilde{\pi}(\rho^\square \otimes \Rsq/\pp_0)$ by Theorem \ref{comp_surj_iso_existence_thm}. Then Lemma \ref{comp_surj_iso_implication_lemma} implies $P \cap \phi^{-1}(\pp_0) \subseteq P_1,$ and therefore $P \cap \phi^{-1}(U_0) \subseteq P_1$. Since $\phi$ is open, $\phi^{-1}(U_0)$ is open dense.
Thus any open $U \subseteq \Spec A[p^{-1}]$ intersects $\phi^{-1}(U_0)$ in an open set, which must contain a point of $P$ since $P$ is dense in $\Spec A[p^{-1}]$ by Lemma \ref{interp_DVRpts_are_dense_in_Ru}.
\end{proof}
 This immediately gives a surjection
\begin{equation}
\label{eq:4}
   \tilde{\Pi}_{\overline{\rho}} \to \tilde{\pi}(\rho^\square \otimes A) 
\end{equation}
by Proposition \ref{lower_bound_prop}. We will show this to be an isomorphism later.

\subsection{Torsion-free modules and formal power series}
\begin{lemma}
\label{r_and_s_lemma}
Let $R$ be a reduced Noetherian ring and $r \in R.$ Then there is some $s \in R$ such that $r+s$ is a regular element (i.e. not a zero divisor), but $rs=0.$
\end{lemma}
\begin{proof}
Recall that since $R$ is reduced the set of zero divisors in $R$ is $\bigcup_{\mfa} \mfa$ where $\mfa$ runs over minimal primes of $R.$ Fix any $r \in R$ and consider the ideal $\bigcap_{r \not\in \pp \in \mathrm{MinSpec} R} \pp$ where $\mathrm{MinSpec}$ denotes the subset of $\Spec$ consisting of minimal primes.\\
For $\mathfrak{q} \le R$ minimal prime ideal satisfying $r \in \mathfrak{q},$ the following holds.
Note that $$\qq \cap \bigcap_{\mathfrak{q} \not\subseteq \pp \in \mathrm{MinSpec} R} \pp= \bigcap_{\pp \in \mathrm{MinSpec} R} \pp=0$$ by reducedness. On the other hand $\bigcap_{\mathfrak{q} \not\subseteq \pp \in \mathrm{MinSpec} R} \pp \neq 0$ by the uniqueness of primary decomposition for the ideal $0$, as the intersection includes no ideal having $\qq$ as its radical. From $r \in \mathfrak{q}$ we have $\bigcap_{\mathfrak{q} \not\subseteq \pp \in \mathrm{MinSpec} R} \pp \le \bigcap_{r \not\in \pp \in \mathrm{MinSpec} R} \pp,$ and so $\qq$ does not contain $\bigcap_{r \not\in \pp \in \mathrm{MinSpec} R} \pp$.
By this argument for all $\qq$ containing $r,$ and using prime avoidance, the ideal $\bigcap_{r \not\in \pp \in \mathrm{MinSpec} R} \pp$ is not contained in the union of prime ideals $\bigcup_{r \in \qq \in \mathrm{MinSpec} R} \qq.$ We choose
$$s \in \left( \bigcap_{r \not\in \pp \in \mathrm{MinSpec} R} \pp \right) \setminus \left( \bigcup_{r \in \qq \in \mathrm{MinSpec} R} \right).$$
Then we have $rs \in \cap_{\pp \in \mathrm{MinSpec} R} \pp=0$ by reducedness. We also have $r+s \in R \setminus \bigcup_{\mathfrak{q} \in \mathrm{MinSpec} R} \mathfrak{q}$ as each minimal prime contain precisely one of $r$ and $s.$ Therefore $r+s$ is a regular element.
\end{proof}

\begin{lemma}
\label{comm_alg_radical_annihilator_lemma}
Let $R$ be a reduced Noetherian ring and $M$ a torsion-free module over $R.$ Then $\ann_R(m)$ is a radical ideal for all $m \in M,$ or equivalently, $M[r^\infty]=M[r]$ for all $r \in R.$
\end{lemma}
\begin{proof}
Take a suitable $s$ for $r$ from Lemma \ref{r_and_s_lemma}. We have 
$$(r+s)M = rM \oplus sM$$
by noting that $(r+s)M \le rM + sM$ and that $rM \cap sM=0$ as it is killed by the regular element $r+s.$
On one hand, we have $rM \cong M/M[r].$ On the other hand, $rM$ has no $r$-torsion since it has no $r+s$-torsion since it is $R$-torsion-free by being a submodule of $M.$ This shows $M[r^i]= M[r]$ for all $i \ge 1.$
\end{proof}

\begin{lemma}
\label{comm_alg_torsionfree_ext_lemma}
Let $R$ be a reduced Noetherian ring, and $M$ an $R$-torsionfree module. Then $M[[u]]$ is $R[[u]]$-torsionfree.
\end{lemma}
\begin{proof}
Let $r=\sum_{i=0}^\infty r_i u^i$ be any regular element of $R[[u]]$ with each $r_i \in R,$ and let $\sum_{i=0}^\infty r_j u^j$ with each $m_j \in M.$ Since $u$ is a regular element of $R,$ it is sufficient to prove that $rm \neq 0$ if $m_0,r_0 \neq 0.$
We first rule out that $r_i m_j=0$ for all $i,j \in \NN.$
Since $M$ is torsion-free and $m_0 \neq 0$ by assumption, $\ann_R(m_0)$ is an ideal consisting of zero divisors, hence contained in a minimal prime $\mfa$ by prime avoidance (using that in a reduced ring, a zero divisor is contained in a minimal prime).
In particular, if $r_i \in \ann_R(m_0)$ for all $i \in \NN,$ then any element of $\ann_R(\mfa)$ annihilates $r.$ As $\ann_R(-) \neq 0$ for a minimal prime in a Noetherian ring, this would make $r$ a zero divisor.\\
Therefore we can consider the lexicographically first $(i,j) \in \NN^2$ such that $r_i m_j \neq 0.$ Consider the coefficient of $u^{i+j}$ in $rm:$
\begin{equation}
\label{eq:3}
    \left( r_0 m_{i+j} + \dots + r_{i-1} m_{j+1} \right) + r_i m_j + \left( r_{i+1} m_{j-1} + \dots + r_{i+j} m_0 \right) \in M
\end{equation}
The first group of terms is $0$ by our minimality hypothesis on $i.$ The last group of terms all lie in $M[r_i]$ (equal to $M[r_i^\infty]$ by Lemma \ref{comm_alg_radical_annihilator_lemma}) by the minimality hypothesis on $j.$ But $r_i m_j \in M \setminus M[r_i^\infty]$ as $m_j \not\in M[r_i^\infty] = M[r_i]$ by assumption. Therefore the expression \eqref{eq:3} is non-zero, and thus $rm \neq 0.$
\end{proof}

\textbf{Remark.} It is not possible to remove the all assumptions on $R$ in this lemma, as shown by the following example. Consider the ring $R=k[[T,T^{1/p},T^{1/p^2},\dots]]/(T)$ and the $R$-module $M=k$ on which $T$ acts by $0.$ Then $M$ is torsion-free over $R$ (as each element of $R$ is either invertible or a zero divisor) but $M[[u]]$ is not torsion-free (or even faithful) over $R[[u]]$ since it is annihilated by the element $\sum_{i=1}^\infty T^{1/p^i} u^i$ which is not a zero divisor in $R[[u]].$

\begin{lemma}
\label{comp_sorminta_lemma}
$$\tilde{\pi}(\rho^\square \otimes A) \cong \tilde{\pi}(\rho^\square) \otimes A.$$
\end{lemma}
\begin{proof}
It is sufficient to show that $\tilde{\pi}(\rho^\square) \otimes A$ is co-Whittaker and torsion-free, in which case the uniqueness property of $\tilde{\pi}$ forces an isomorphism, as they agree on all irreducible components. The co-Whittaker property is preserved by change of coefficients by Lemma \ref{quotientofcowhittaker_lemma}, and it is torsion-free by Lemma \ref{comm_alg_torsionfree_ext_lemma}.
\end{proof}

\subsection{The isomorphism}
We now show the surjection
$$\tilde{\Pi}_{\overline{\rho}} \to \tilde{\pi}(\rho^\square \otimes A),$$
shown to exist in \eqref{eq:4}, to be an isomorphism by constructing a surjection in the opposite direction.

For each $x \in P$ we have a surjection
$$W^\square \otimes_A \kappa(x) \to \tilde{\pi}(\rho^\square \otimes A) \otimes_A \kappa(x) \to \tilde{\pi}(\rho^\square \otimes \kappa(x))$$
where the second map exists by Corollary \ref{comp_surj_iso_existence_thm} and Lemma \ref{comp_surj_iso_implication_lemma}.

This gives rise to a surjection $$\gamma_x : \tilde{\pi}(\rho^\square \otimes A) \to \tilde{\pi}(\rho^\square \otimes \kappa^+(x))$$ as in Proposition \ref{interp_surjection_to_pirhox_prop}. We now have maps
$$W^\square \to \tilde{\pi}(\rho^\square \otimes A) \xrightarrow{\gamma} \prod_{x \in P} \tilde{\pi}(\rho^\square \otimes \kappa^+(x)).$$

Using the uniqueness in Proposition \ref{interp_surjection_to_pirhox_prop}, we have 
$$\ker \left( W^\square \to \tilde{\pi}(\rho^\square \otimes A) \xrightarrow{\gamma_x} \tilde{\pi}(\rho^\square \otimes \kappa^+(x)) \right) = \ker \beta_x$$
for our $\beta_x$ chosen earlier. Considering images inside the product, we obtain a surjection
$$\tilde{\pi}(\rho^\square \otimes A) \to \tilde{\Pi}_{\overline{\rho}}.$$
We wish to prove that this is an isomorphism with its inverse given by \eqref{eq:4} up to an element of $A^\times.$ Indeed, their composition is a surjective endomorphism of the co-Whittaker $A[G]$-module $\tilde{\pi}(\rho^\square \otimes A),$ which belongs to $A^\times$ by \cite[Proposition 6.2]{helm2016whittaker}, and is therefore an isomorphism.
\begin{thm}
\label{equality_thm}
There are isomorphisms
$$\tilde{\pi}(\rho^\square \otimes A) \xrightarrow{\sim} \tilde{\Pi}_{\overline{\rho}}$$
and 
$$\tilde{\pi}(\rho^\square) \xrightarrow{\sim} \tilde{\Pi}_{\overline{\rho}}/u \tilde{\Pi}_{\overline{\rho}}.$$
\end{thm}
\begin{proof}
The first isomorphism results from the immediately preceding discussion. The second isomorphism then follows from the first and Lemma \ref{comp_sorminta_lemma}.
\end{proof}

We obtain the following corollaries.
\begin{cor}
$\tilde{\Pi}_{\overline{\rho}}$ is torsion-free over $A.$
\end{cor}

\begin{cor}
\label{bigPi_embeds_over_P1_cor}
Recalling the definition of $\tilde{\Pi}_{\overline{\rho}}$  as a submodule of a product over $P$, the map induced by the natural projection $\tilde{\Pi}_{\overline{\rho}} \to \prod_{\qq \in P_1} \tilde{\pi}(\rho^\square \otimes \kappa^+(\qq))$ is injective.
\end{cor}
\begin{proof}
This is immediate from Theorem \ref{equality_thm} and Lemma \ref{Wsq_pi_injects_to_product_lemma}.
\end{proof}

We conclude with the proof of Theorem \ref{evenstronger_thm}.
\begin{proof}
Theorem \ref{other_main_result_thm} and our assumption together give a surjection in one direction, so it is enough to find a surjection from $\tilde{\Pi}^{\mathrm{DVR}}(\rho^\square \otimes \Rsq/\pp) \otimes \kappa(\pp)$ to $\tilde{\pi}(\rho^\square \otimes \kappa(\pp)).$  Let $\pp' \in P$ be a prime lying over $\pp.$ We then have $\kappa(\pp)=\kappa(\pp').$ 
Let $S \in \mathrm{QDVR}_k$ be any proper quotient of $\kappa^+(\pp).$ Then we have a composite map
$$\Rsq/\pp \hookrightarrow \kappa^+(\pp') \to S,$$
and $\rho^\square \otimes \kappa^+(\pp')$ is a lift of $\rho^\square \otimes \Rsq/\pp \otimes S.$
Recalling the definition of $\tilde{\Pi}^{\mathrm{DVR},u}(\rho^\square \otimes \Rsq/\pp)$ as a universal co-Whittaker cover, it must admit a surjection 
$$\tilde{\Pi}^{\mathrm{DVR},u}(\rho^\square \otimes \Rsq/\pp) \to \tilde{\pi}(\rho^\square \otimes \kappa^+(\pp'))$$ because it admits a surjection to all $\tilde{\pi}(\rho^\square \otimes \kappa^+(\pp')) \otimes_{\kappa^+(\pp)} S.$\\
The above surjection must factor through $\tilde{\Pi}^{\mathrm{DVR},u}(\rho^\square \otimes \Rsq/\pp) \otimes_{\Rsq[[u]]} \kappa^+(\pp')$ as its target is a $\kappa^+(\pp')$-module.
Localising both, we obtain a surjection
$$\tilde{\Pi}^{\mathrm{DVR},u}(\rho^\square \otimes \Rsq/\pp) \otimes_{\Rsq[[u]]} \kappa(\pp') \to \tilde{\pi}(\rho^\square \otimes \kappa(\pp')).$$
Noting that the map
$$\tilde{\Pi}^{\mathrm{DVR}}(\rho^\square \otimes \Rsq/\pp) \otimes_{\Rsq} \kappa(\pp) \to \tilde{\Pi}^{\mathrm{DVR},u}(\rho^\square \otimes \Rsq/\pp) \otimes_{\Rsq[[u]]} \kappa(\pp')$$
is an isomorphism (as $\kappa(\pp)=\kappa(\pp')$), the last surjection becomes
$$\tilde{\Pi}^{\mathrm{DVR}}(\rho^\square \otimes \Rsq/\pp) \otimes_{\Rsq} \kappa(\pp) \to \tilde{\pi}(\rho^\square \otimes \kappa(\pp)).$$
\end{proof}



\bibliography{semicontinuity}

\end{document}